%% file: main.tex
\documentclass[reqno, 11pt, american]{amsart}
\title{High-energy eigenfunctions of point-perturbations of the Laplacian}

\author[S. Verdasco]{Santiago Verdasco}
\address{M$^2$ASAI. Universidad Politécnica de Madrid, ETSI Navales, Avda. de la Memoria, 4, 28040, Madrid, Spain.}
\email{santiago.verdasco@upm.es}

\usepackage[utf8]{inputenc}
\usepackage[T1]{fontenc}
\usepackage{geometry}
\usepackage[sorting = nyt, sortcites = true, giveninits = true, isbn = false]{biblatex}
\AtEveryBibitem{\clearfield{number}\clearfield{pages}}
\renewbibmacro{in:}{}

\usepackage{amssymb}
\usepackage{mathrsfs}
\usepackage{mathtools}
\usepackage{amsthm}
\usepackage{keytheorems}
\usepackage{dsfont}

\usepackage[dvipsnames]{xcolor}
\usepackage[shortlabels]{enumitem}
\usepackage{hyperref}
\hypersetup{hypertexnames=false} 
\usepackage{zref-clever}
\zcsetup{cap=true}

\theoremstyle{plain}
\newkeytheorem{Theorem}[parent=section]
\newkeytheorem{Lemma}[sibling=Theorem]
\newkeytheorem{Corollary}[sibling=Theorem]
\newkeytheorem{Prop}[name=Proposition, sibling=Theorem]

\theoremstyle{definition}
\newkeytheorem{Definition}[sibling=Theorem]
\newkeytheorem{Example}[sibling=Theorem]
\newkeytheorem{Remark}[sibling=Theorem]
\newkeytheorem{Claim}[numbered=false]

\counterwithin{equation}{section} 

\input{Repository/commands}

\def \sqLap{\sqrt{\Lap}}
\def \SDM{semiclassical defect measure}
\def \SDMs{semiclassical defect measures}

\def \matA{\mathbb{A}}

\NewDocumentCommand{\sympf}{O{} m m O{}}{{#1\{ #2 \, , #3 #1\}}_{#4}}

\begin{document}

\begin{abstract}
In this paper, we explore the high-frequency properties of eigenfunctions of point perturbations of the Laplacian on a compact Riemannian manifold. These systems cannot be obtained as the quantization of a classical Hamiltonian, as the effect of the perturbation amounts to prescribing certain boundary conditions on a discrete set of points. We are interested in understanding to what extent the high-frequency behavior of eigenfunctions is governed by the global dynamics of the geodesic flow in the manifold (the classical flow corresponding to the unperturbed Laplacian). We prove that as soon as the Laplacian is perturbed by a finite set of point scatterers satisfying a \emph{non-self-focality} condition, namely, that the family of geodesics starting from this set and coming back to it has zero measure, semiclassical measures corresponding to high-frequency sequences of eigenfunctions are invariant under the geodesic flow. Invariance may fail when the non-self-focality condition does not hold, as is shown in a companion article [DOI:\href{https://doi.org/10.48550/arXiv.2601.19701}{10.48550/arXiv.2601.19701}]. Our results are based on a quasimode construction that requires improved estimates on the spectral function of the Laplacian on the set of scatterers.
\end{abstract}

\maketitle

\section{Introduction}
\label{Sec: Introduction}

The analysis of high-frequency eigenfunctions of the Laplace-Beltrami operator (or Laplacian) on a smooth, compact Riemannian manifold $(M,g)$ has been the subject of intense research in the past sixty years from very different points of view; see, for instance, the monograph \cite{Zelditch2017}. 

Throughout this article $\Lap$ will denote the positive Laplacian on $(M,g)$. When $V\in L^\infty(M,\R)$, the Schrödinger operator $\Lap {}+ V$ is self-adjoint on $L^2(M)$,  and its spectrum is a discrete sequence of real numbers that tends to infinity:
\begin{equation*}
    \Spec( \Lap {}+ V )=\{ \lambda_n^2 \colon n \in \N \} \ .
\end{equation*}
The corresponding eigenfunctions solve:
\begin{equation}\label{Eq:eigenfdef}
    (\Lap {}+ V) u_{\lambda_n} = \lambda_{l}^2 u_{\lambda_n},\qquad  \norm{u_{\lambda_n}}[L^2(M)] = 1 \ .
\end{equation}
When studying the structure of those regions of $M$ where the sequence of densities $(\abs{u_{\lambda_n}(x)}^2\D{x})_{n \in \N}$ may concentrate, one usually considers the corresponding sequence of \emph{microlocal lifts} of these densities, the Wigner distributions. If $\Opp_{h}$ for $h > 0$ stands for the semiclassical Weyl quantization on $M$, the \emph{Wigner distribution} of a function $u \in L^2(M)$ is
\begin{equation} \label{Eq: Wigner dist def}
    W_{h}[u](a) \coloneqq \ip{u}{\Op[h]{a} u}[L^2(M)] \ , \qquad a \in \CinfK(T^*M) \ .
\end{equation}
Thanks to the Calderon-Villantcourt theorem, $W_{h}[u]$ is, in fact, a distribution on $T^*M$ for every $u \in L^2(M)$, $h > 0$. The accumulation points of $(W_{h_n}[u_{\lambda_n}])_{n \in \N}$, for a sequence of eigenfunctions $(u_{\lambda_n})_{n \in \N}$ satisfying \eqref{Eq:eigenfdef} and $h_n \coloneqq (\lambda_{n})^{-1} \to 0^+$, are always probability measures on $T^*M$. These are called \emph{\SDMs{}} of $\Lap {}+ V$, and capture concentration and oscillation-type phenomena developed by the sequence $(u_n)_{n \in \N}$ of high-frequency eigenfunctions.

Characterizing the set of \SDMs{} of $\Lap {}+ V$ and, in particular, clarifying how this set depends on the global geometry of $M$ and the properties of the perturbation $V$ is, in general, a very hard problem. However, as soon as $V\in\Cinf(M,\R)$, it is well known that any \SDM{} $\mu$ of $\Lap {}+ V$ belongs to $\mathcal{P}_{\mathrm{inv}}(S^*M)$, the set  of probability measures on  $T^*M$ satisfying:
\begin{enumerate}
    \item $\supp \mu \subseteq S^*M = \{ (x, \xi) \in T^*M \colon \abs{\xi}_x = 1\}$.
    \item $\mu$ is invariant under the geodesic flow $\{\phi_t\}_{t \in \R}$ on $S^*M$, 
    \begin{equation*}
        (\phi_t)_* \mu = \mu,\quad \text{for all $t \in \R$} \ .
    \end{equation*}
\end{enumerate}
However, except for very few cases (when $M$ is a compact, rank-one symmetric space and $V = 0$, see \cite{JakobsonZelditch1999, Macia2008, AzagraMacia2010}), these two properties fail to completely characterize the set of \SDMs{} of $\Lap {}+ V$. This set is, in general, strictly smaller than $\mathcal{P}_{\mathrm{inv}}(S^*M)$.  If $M$ has negative sectional curvature and $V=0$, \SDMs{} have positive entropy \cite{AnantharamanNonnenmacher2007Anosovman, Anantharaman2008, Riviere2010} (this prevents uniform measures supported on closed geodesics from being a \SDM{}) and are of full support  when $d = 2$ \cite{DyatlovJin2018, DyatlovJinNonnenmacher2022}. The \emph{quantum unique ergodicity} conjecture \cite{ColindeVerdiere1985, RudnickSarnak1994, Snirelman1974, Zelditch1987} states that the set of \SDMs{} of $\Lap$ reduces to the Liouville measure when $M$ is a surface of constant negative curvature. This was proved under an additional symmetry hypothesis in \cite{Lindenstrauss2006}. When $M$ is a flat torus, the set of \SDMs{} of $\Lap$ is  also strictly smaller than the set of invariant measures \cite{Jakobson1997, BurqGermainSorellaZhu2025, BurqZhu2025}.

The presence of a bounded smooth perturbation $V$ induces an additional set of restrictions for an invariant measure being a \SDM{} of $\Lap {}+ V$. This was investigated in \cite{MaciaRiviere2016, MaciaRiviere2019, Macia2021} for perturbations on the sphere or Zoll manifolds, or \cite{AnantharamanMacia2014, AnantharamanFer-KammererMacia2015, AnantharamanLeautaudMacia2016, MaciaRiviere2018} when the geodesic flow is completely integrable.  In these situations, \SDMs{} enjoy additional invariance properties under certain flows defined in terms of $V$.

In this paper, we address these questions in the case of a certain class of singular, unbounded perturbations, namely point-perturbations of $\Lap$. These perturbations are used to model Dirac delta potentials $\delta_{q}$ centered on a finite family of points $q \in Q \subseteq M$ that are coupled according to some self-adjoint boundary conditions on $Q$; these perturbations are non-trivial only for dimensions $d = 1, 2, 3$. In the Mathematical Physics literature, these are usually denoted as 
\[
\Lap {}+ {\textstyle \sum_{q, p \in Q}} \ \delta_{p} \matA \delta_{q} \ .
\]
However, we opt for an alternative notation, $\Lap_{Q, L}$, where the self-adjoint extension is parameterized by a certain Lagrangian-Grassmannian. The connection between these two approaches is described in detail in \zcref{subSec: Point-pert of Lap. L and matA}.

Let $Q \subset M$ denote the set of points at which the deltas are supported, suppose $d \coloneqq \dim(M) = 2,3$, and $N \coloneqq \card Q$. For every Lagrangian subspace $L$ of $\C[N] \times \C[N]$, equipped with the standard symplectic form (see \eqref{Eq: sympl coord of dom(A^*)/dom(A)}), $\Lap_{Q, L}$ denotes a unique self-adjoint extension of $\Lap|_{\CinfK(M \setminus Q)}$. Point-perturbations of $\Lap$ are singular perturbations in the sense that they are no longer pseudodifferential operators on $M$; therefore, it is not clear what their classical analogs are. They are also known in the literature as point scatterers \cite{Ueberschaer2014}, or delta potentials \cite{Hillairet2002}; the Lagrangian subspace $L$ represents \emph{coupling parameters} between the scatterers. These operators $\Lap_{Q, L}$ are sometimes called \emph{pseudo-Laplacians} \cite{ColindeVerdiere1982}. These point-perturbations can also be understood as conical singularities of angle $2\pi$ on any given point of a manifold (see \cite{Hillairet2008}). 

The spectrum of these operators is still discrete and tends to infinity; in fact, point-perturbations behave spectrally as finite-rank perturbations: eigenfunctions of $\Lap$ that vanish on $Q$ are eigenfunctions of $\Lap_{Q, L}$ with the same eigenvalue. The perturbation creates families of \emph{new eigenfunctions} that turn out to have a Green-function type singularity at the points in $Q$. These are the ones we are most interested in, since in generic situations, no eigenfunction of $\Lap$ is an eigenfunction of $\Lap_{Q, L}$. \footnote{This is due to the fact that on a generic closed Riemannian manifold, every eigenvalue of $\Lap$ is simple \cite{Uhlenbeck1976}, thus the union of the nodal sets of the sequence of eigenfunctions is meager in $M$. Therefore, for a residual set of $q \in M$, no eigenfunction of $\Lap$ survives after a point-perturbation on $q$ has been made.}

Given a sequence $(\Lap_{Q, L_n})_{n \in \N}$ of point-perturbations of $\Lap$ on a fixed finite set $Q \subseteq M$, for every $n \in \N$ let $u_n \in \dom(\Lap_{Q, L_n}) \subseteq L^2(M)$ and $h_n > 0$ be such that 
\begin{equation} \label{Eq: eigenf main theorem}
    h_n^2 \Lap_{Q, L_n} u_n = u_n \ , \qquad \norm{u_n}[L^2(M)]^2 = 1 \ , \qquad \lim_{n \to \infty} h_n = 0 \ .
\end{equation}
Our main theorem shows that as soon as $Q$ is \emph{non-self-focal} (this is defined right below), any \SDM{} associated with a sequence $(u_n)_{n \in \N}$ of eigenfunctions satisfying \eqref{Eq: eigenf main theorem} belongs to $\mathcal{P}_{\text{inv}}(S^*M)$. 

A discrete subset $Q\subset M$ is said to be \emph{non-self-focal} provided that, for every $q, p \in Q$,
\begin{enumerate}
    \item $q$ is \emph{non-self-focal}, \textit{i.e.} the loopset
    \begin{equation} \label{Eq: self-focal dir set q}
        \mathcal{L}_{q} \coloneqq \{ \xi \in S_{q}^*M \colon \ \exists \, t > 0 \ \text{s.t.} \ \exp_{q}(t\xi) = q \}
    \end{equation}
    has zero measure with respect to the measure induced by $g(q)$ on $T_{q}^*M$.
    \item $(q,p)$ are \emph{mutually non-focal},  \textit{i.e.} the set
    \begin{equation} \label{Eq: mutually focal focal dir set qp}
        \mathcal{L}_{q,p} \coloneqq \{ \xi \in S_{q}^*M \colon \ \exists \, t > 0 \ \text{s.t.} \ \exp_{q}(t\xi) = p \}
    \end{equation}
    has zero measure with respect to the measure induced by $g(q)$ on $T_{q}^*M$.
\end{enumerate}

\begin{Remark} \label{Rmk: Cartan-Hadamard thm}
Closed Riemannian manifolds with the property that \emph{every finite} $Q\subset M$ is non-self-focal include: 
    \begin{enumerate}
        \item Manifolds whose sectional curvatures are all less or equal to zero. In  this case, every set $\mathcal{L}_{q}$ and $\mathcal{L}_{q, p}$ is countable, thanks to the Cartan-Hadamard theorem.
        \item Products of manifolds of positive dimension equipped with the product Riemannian metric. This follows from the fact that the exponential map in the product coincides with the product of the exponential maps of its factors. These include manifolds with non-negative sectional curvatures, such as $\S[2]\times\S[1]$. 
    \end{enumerate}
    When $Q$ consists of one point, $Q$ being non-self-focal is a generic condition with respect to the metric. In fact, a stronger property is true: for any given closed manifold $M$, $\dim(M) \geq 2$, the set of Riemannian metrics $g$ on $M$ for which $\mathcal{L}_{q}$ has zero measure for all $q \in M$ is residual in the Whitney $\Cinf$-topology, see \cite[Lemma 6.1]{SoggeZelditch2002}.
\end{Remark}

We can now state the main result of this article.

\begin{Theorem} \label{Thm: main theorem}
    Let $(M,g)$ be a closed Riemannian manifold of dimension $2$ or $3$ with positive Laplacian $\Lap$. Let $(\Lap_{Q, L_n})_{n \in \N}$ be a sequence of point-perturbations of $\Lap$ on a fixed finite set $Q \subseteq M$. For every $n \in \N$, let $u_n \in D(\Lap_{Q, L_n}) \subseteq L^2(M)$ and $h_n > 0$ satisfy \eqref{Eq: eigenf main theorem}. Any \SDM{} $\mu$ associated to the sequence $(u_n)_{n \in \N}$ enjoys the following properties:
    \begin{enumerate}
        \item $\mu$ is a probability measure, and $\supp(\mu) \subseteq S^*M$.
        \item If $Q$ is non-self-focal, then $\mu$ is invariant under geodesic flow.
    \end{enumerate}
\end{Theorem}

\begin{Remark}
    The non-self-focality condition in \zcref{Thm: main theorem} is sharp. In \cite{Verdasco2026spheres} we study high-energy eigenfunctions of sequences of point-perturbations of $\Lap$ on a fixed set $Q$ on the spheres $\S[d]$, $d = 2,3$; a setting where any finite subset $Q$ fails to be non-self-focal. In that paper, we prove that if $Q$ contains a pair of antipodal points, there exists a sequence $(u_n)_{n \in \N}$ as in \zcref{Thm: main theorem} that produces a non-invariant \SDM{}. In fact, there exists at least a continuous one-parameter family of non-invariant measures attainable this way.
\end{Remark}

\begin{Remark}
    One can prove through a standard microlocal argument (see \zcref{Thm: main theorem spectral cond}) that semiclassical defect measures $\mu$ in \zcref{Thm: main theorem} are invariant under geodesic flow outside $Q$ in the following sense: if $\pi(x, \xi) \coloneqq x$, for every $a \in \CinfK(T^*M)$,
    \[
    \int_{T^*M} \sympf{a}{H} \mu = 0 \qquad \text{if} \ \pi(\supp a) \subseteq M \setminus Q \ .
    \]
\end{Remark}

\zcref{Thm: main theorem} is implied by the more general \zcref{Thm: main theorem spectral cond}, in which the non-self-focality condition on $Q$ is replaced by a, in principle, more general point-wise estimate on "wide" quasimodes of $\sqLap$. The proof that the non-self-focality condition on $Q$ implies the point-wise estimate is the content of \zcref{Lemma: gamma for density improvement}.

Point-perturbations of $\Lap$ are limiting models for potentials with small support around some point. The interest in this kind of model was raised by S{\v{e}}ba \cite{Seba1990} on Dirichlet billiards, as a way to obtain quantum models with chaotic features (distribution of level spacing similar to GOE) whose underlying classical model is integrable. For a single point potential, the family of non-trivial self-adjoint perturbations of $\Lap$ is parametrized by $\alpha \in \R$, known as \emph{coupling constant}. S{\v{e}}ba considered what is called \emph{weak coupling}, for which the coupling constant is fixed as the energy increases. Later, \cite{Shigehara1993} proposed the notion of \emph{strong coupling}, where the coupling constant $\alpha$ depends on the energy, with the intention of forcing the new eigenvalues to stay away from the old ones (see \cite{Ueberschaer2014, KurlbergLesterRosenzweig2023} for additional details). The spectral statistics of point-scatterers have been thoroughly studied in the past twenty years, see for instance \cite{AissiouHillairetKokotov2012, BogomolnyGerlandSchmit2001, BogomolnyGiraudSchmit2002, Hillairet2002, RahavFishman2002, RudnickUeberschaer2012, Ueberschaer2012, RudnickUeberschaer2014, FreibergKurlbergRosenzweig2017} among many others.

The study of \SDMs{} is much less developed, except when the underlying manifold is a torus $\T[d] \coloneqq \R[d]/\Z[d]$. In this case, the spectra of point scatterers and the asymptotic properties of their eigenfunctions can be studied using tools from analytic number theory. 

In the case of one or two point-scatterers, very precise results on semiclassical measures on $\T[d]$, $d = 2, 3$, that go beyond invariance under the geodesic flow, have been obtained. First, in \cite{RudnickUeberschaer2012}, the authors showed that a density-one subsequence of new eigenfunctions of $\Lap$ with a single point-potential equidistributes in configuration space in $\T[2]$. In \cite{Yesha2013}, the previous result was extended to $\T[3]$ for the full sequence of new eigenfunctions. Later, these results were improved to equidistribution in phase space for a density-one subsequence of new eigenfunctions for a single scatterer on $\T[2]$ \cite{KurlbergUeberschaer2014} and on $\T[3]$ \cite{Yesha2015}. For one-delta perturbations, \cite{KurlbergRosenzweig2017,KurlbergLesterRosenzweig2023} proved that there exist zero density sequences of new eigenfunctions on $\T[2]$ and $\T[3]$ (in the weak and strong coupling regime) whose \SDMs{} agree with Lebesgue measure on the configuration variable but are singular with respect to the Lebesgue measure in the momentum variable; this latter phenomenon is known as \emph{superscarring}. The case of two-dimensional irrational tori was addressed in \cite{KurlbergUeberschaer2017}; superscarring is shown to take place for a density-one sequence of new eigenfunctions in the weak coupling regime. The case of two deltas in $\T[2]$ was investigated in \cite{Yesha2018}. Equidistribution in the position variable is observed for a density-one subsequence of new eigenfunctions under some diophantine condition for the scatterers location. The closely related case of regular polygons has been addressed in \cite{MarklofRudnick2012}.

The general strategy in these works is to find sequences of real numbers for which the corresponding new eigenfunction can be nicely approximated by a quasimode of $\Lap$ that enjoys the properties of interest (see \cite{KeatingMarklofWinn2010}). The width and the properties of these quasimodes are limited by the distribution of $\Spec(\Lap)$ with multiplicity around a sequence of new eigenvalues. This is related to the analysis of the distribution of lattice points on annuli in $\R[2]$ and $\R[3]$. 

Our approach to the general case follows the same broad strategy: find good enough quasimodes of new eigenfunctions to infer properties of \SDMs{}. In order to gain control on the width of the quasimodes, the lack of an explicit representation for eigenvalues and eigenfunctions of $\Lap$ is overcome by resorting to spectral geometric results on the asymptotics of the kernel of the spectral function of $\sqLap$.

\subsection*{Outline of the paper}
\label{subSec: Introduction. Outline}

In \zcref{Sec: Point-perturbations of Laplacian} we define several point perturbations of $\Lap$ in a general closed manifold, following the scheme of \cite{Hillairet2010} and \cite{HillairetKokotov2012}, and relate their spectrum and eigenfunctions to the spectrum and eigenfunctions of $\Lap$. \zcref{Thm: resolvent identity Lap LapL} provides a resolvent identity between $\Lap$ and a point perturbation of it. In \zcref{Sec: Approximation by quasimodes} we construct quasimodes for linear combinations of high-energy Green's functions of $\Lap$, the new eigenfunctions of a point-perturbation. Finally, in \zcref{Sec: Properties of SDM} we use these approximations to show that any \SDM{} associated with a sequence of eigenfunctions satisfying \eqref{Eq: eigenf main theorem} is always supported in $S^*M$, and is invariant under the geodesic flow provided some point-wise asymptotics of the spectral function of $\sqLap$ on $Q$ hold.

\subsection*{Acknowledgments}

The author is indebted to Fabricio Macià as this problem was suggested by him. I would like to thank him for his help in the elaboration of the paper and his encouragement to obtain \zcref{Thm: main theorem} in its present state. I would also like to thank Víctor Arnaiz, Jared Wunsch, and Yuzhou Joey Zou for helpful discussions on the topic.
This research has been supported by grants PID2021-124195NB-C31 and PID2024-158664NB-C21 from Agencia Estatal de Investigación (Spain), and by grant VPREDUPM22 from Programa Propio UPM.

\section{Point-perturbations of Laplacian}
\label{Sec: Point-perturbations of Laplacian}
In this section, we briefly introduce the elements of the theory of point-perturbations that are needed to prove our results.

\subsection{Description of point-perturbations}
\label{subSec: Point-pert of Lap. Description}

Recall that $\Lap$ denotes the (positive) Laplacian on $M$ with domain $\dom(\Lap) = H^2(M)$. From now on, $Q$ will denote a finite set of points of $M$ with $N \coloneqq \card Q$; let $A$ denote the closure of the symmetric operator $\Lap|_{\CinfK(M\setminus Q)}$. Point perturbations of $\Lap$ are self-adjoint extensions of the closed, symmetric operator $A$.

These self-adjoint extensions are characterized via von Neumann's theory. We briefly recall the main points, following the presentation in \cite[Sec. 3.2]{Hillairet2010} and \cite[Sec. 3]{HillairetKokotov2012}. The domain of $A^*$, $\dom(A^*)$, is a Hilbert space with respect to the graph inner-product $\ip{u}{v}[A^*] = \ip{u}{v} + \ip{A^*u}{A^*v}$, and $\dom(A)$ is a closed subspace of it. One may define the following skew-Hermitian form on $\dom(A^*)$:
\[
\omega( u, v) = \ip{u}{A^*v} - \ip{A^*u}{v} \qquad u, v \in \dom(A^*) \ .
\]
The restriction of this form to $\dom(A^*) / \dom(A)$ is non-degenerate because $\omega(u, \cdot) \equiv 0$ if and only if $u \in \dom(A)$, thanks to $A^{**} = A$ \footnote{Observe that $(A^*)^2 = - \Id$ on $\dom(A)^{\perp}$.}.

Given a closed subspace $L$ of $\dom(A^*)/\dom(A)$, define $A_L$ as the restriction of $A^*$ to $\dom(A_L) \coloneqq \pi^{-1}(L)$, where $\pi \colon \dom(A^*) \to \dom(A^*)/\dom(A)$ is the quotient map.
\begin{Theorem} \label{Thm: AL self-adjoint iff L Lagrangian}
    The operator $A_L$ is self-adjoint if and only if $L$ is Lagrangian.
\end{Theorem}

In our setting, one can prove that
\[
\dom(A) = \cl{\CinfK(M\setminus Q)}^{H^2} = \begin{cases}
    \{ u \in H^2(M) \colon u(q) = 0 \, , \ u'(q) = 0 \, , \ \forall \ q \in Q \} & \text{if $d = 1$} \ , \\
    \{ u \in H^2(M) \colon u(q) = 0, \ \forall \ q \in Q \} & \text{if $d = 2, 3$} \ , \\
    H^2(M) & \text{if $d \geq 4$} \ .
\end{cases}   
\]
Since $\Lap$ is a self-adjoint extension of $A$, one finds that
\[
\dim( \dom(A^*) / H^2(M) ) = \dim(H^2(M) / \dom(A)) = \begin{cases}
    2N & \text{if $d = 1$} \ , \\
    N & \text{if $d = 1$} \ , \\
    0 & \text{if $d = 1$} \ .
\end{cases}
\]
This implies that if $d \geq 4$, then $\Lap$ is the only self-adjoint extension of $A$, and no point interactions can be defined. The case $d = 1$ will not be addressed here (see \cite{GadellaGlasserNieto2010} for instance).

We focus on the case $d = 2,3$. We construct a coordinate map from $\dom(A^*)/\dom(A)$ to $\C[N] \times \C[N]$ that captures the the boundary conditions on the points of $Q$. We show below that this map is in fact a linear symplectomorphism. Let $\{U_q\}_{q \in Q}$ be a family of pair-wise disjoint open coordinate neighborhoods for each $q$. For each $q \in Q$, let $\rchi_{q,1} \in \CinfK(U_q; \R)$ with $\rchi_{q, 1} (q) = 1$, and let $\rchi_{q, 2} \in \CinfK(U_q; \R)$ be such that $\supp \rchi_{q, 1} \subseteq \rchi_{q, 2}^{-1}(1)$. Lastly, let $F_q$ be the fundamental solution to $\Lap$ on $U_q$ with Dirichlet boundary conditions, that is, $F_q \in L^2(U_q) \subseteq L^2(M)$ is the unique $L^2(U_q)$-function such that $\ip{F_q}{\Lap u}[L^2(U_q)] = u(q)$ for all $u \in \CinfK(U_q)$. Observe that $F_q \in \Cinf(U_q \setminus \{q\})$. One may check that
\begin{equation*}
    H^2(M) = \dom(A) \oplus \Span{ \rchi_{q, 1} \colon q \in Q} \ , \qquad \dom(A^*) = H^2(M) \oplus \Span{ \rchi_{q, 2} F_{q} \colon q \in Q} \ .
\end{equation*}
Therefore, $u \in \dom(A^*)$ if and only if there exists $u_0 \in \dom(A)$ and coefficients $(a_{1}, a_{2}) \in \C[N] \times \C[N]$ such that
\begin{equation} \label{Eq: u0 + c rchi + s Gq}
    u = u_0 + \sum_{q \in Q} a_{q, 1} \rchi_{q, 1} + a_{q, 2} \rchi_{q, 2} F_{q} \ .
\end{equation}
The skew-Hermitian form $\omega$ has a very nice form in this coordinates $(u_0, a_{1}, a_{2})$
\begin{Lemma} \label{Lemma: symplectic form on dom(A*)}
    If $u = u_0 + \sum_{q \in Q} a_{q, 1} \rchi_{q, 1} + a_{q, 2} \rchi_{q, 2} F_{q}$ and $v = v_0 + \sum_{q \in Q} b_{q, 1} \rchi_{q, 1} + b_{q, 2} \rchi_{q, 2} F_{q}$, then
    \[
    \omega(u,v) = \sum_{q \in Q} (\conj{a_{q, 1}} b_{q, 2} - \conj{a_{q, 2}} b_{q, 1}) \ .
    \]
\end{Lemma}
\begin{proof}
    Observe that $\rchi_{q, 1}$, $A^*\rchi_{q, 1}$, $\rchi_{q, 2}$, and $\rchi_{q, 2} F_{q}$, are all real-valued. Since $U_{q} \cap U_{p} = \varnothing$ on $M$ if $p \neq q$, and $\omega(u_0, \fdot) \equiv 0$ if $u_0 \in \dom(A)$, it is straight-forward to check
    \[
    \omega(u,v) = \sum_{q \in Q} (\conj{a_{q, 1}} b_{q, 2} - \conj{a_{q, 2}} b_{q, 1}) \Big[ \ip{\rchi_{q, 1}}{A^* (\rchi_{q, 2} F_{q})}[L^2(M)] - \ip{ \rchi_{q,2} F_{q}}{A^* \rchi_{q, 1}}[L^2(M)] \Big] \ .
    \]
    We compute $A^*\rchi_{q,1}$ and $A^* (\rchi_{q, 2} F_q)$.
    
    On the one hand, since $\rchi_{q, 1} \in \Cinf(M) \subseteq \dom(\Lap)$, and $\Lap$ is a self-adjoint extension of $A$, then for all $u \in \CinfK(M \setminus Q) \subseteq \dom(\Lap)$
    \[
    \ip{\rchi_{q, 1}}{A u}[L^2(M)] = \ip{\rchi_{q, 1}}{\Lap u}[L^2(M)] = \ip{\Lap \rchi_{q, 1}}{u}[L^2(M)] \ ,
    \]
    therefore $A^*\rchi_{q, 1} = \Lap \rchi_{q, 1}$.

    On the other hand, using that $\Lap f = -\div(\nabla f)$ for any $f \in \Cinf(M)$, we find that for any $u \in \CinfK(M \setminus Q)$,
    \[
    \rchi_{q, 2} \Lap u = \Lap(\rchi_{q, 2} u) + 2 \ip{\nabla u}{\nabla \rchi_{q, 2}}[g] - u \Lap \rchi_{q, 2} \ . 
    \]
    Hence, for all $u \in \CinfK(M \setminus Q) \subseteq \dom(\Lap)$,
    \begin{align*}
        \ip{\rchi_{q, 2} F_q}{A u}[L^2(M)] & = \int_{U_q} F_q \ [\rchi_q \Lap u] \\
        & = \int_{U_q} F_q \ \Big[ \Lap(\rchi_{q, 2} u) + 2 \ip{\nabla u}{\nabla \rchi_{q, 2}}[g] - u \Lap \rchi_{q, 2} \Big] \ .
    \end{align*}
    Since $\rchi_{q, 2} u \in \CinfK(U_q)$, then
    \[
    \int_{U_q} F_q \ \Lap(\rchi_{q, 2} u) = [\rchi_{q, 2} u](q) = 0 \ .
    \]
    Since $2 F_q \nabla \rchi_{q, 2}$ is a smooth vector field on $M$ because $\nabla \rchi_{q, 2} \equiv 0$ on a neighborhood of $q$, then
    \[
    \int_{U_q} \ip{ 2 F_q \nabla \rchi_{q, 2} }{ \nabla u}[g] = \int_{U_q} -\mathop{\mathrm{div}}( 2F_q \nabla \rchi_{q,2} ) u \ .
    \]
    Summing up, we find that
    \begin{equation} \label{Eq: aux pag6}
    \ip*{\rchi_{q, 2} F_q}{A u}[L^2(M)] = \ip*{ - \mathop{\mathrm{div}}( 2F_q \nabla \rchi_{q,2} ) - F_{q} \Lap \rchi_{q, 2} }{u}[L^2(M)] \qquad \forall \, u \in \CinfK(M \setminus Q) \ ,
    \end{equation}
    thus
    \[
    A^*(\rchi_{q, 2} F_{q}) = - \mathop{\mathrm{div}}( 2F_q \nabla \rchi_{q,2} ) - F_{q} \Lap \rchi_{q, 2} \ .
    \]

    Now, taking into account that $\nabla \rchi_{q, 2} = 0$ and $\Lap \rchi_{q,2} = 0$ on $\supp \rchi_{q, 1}$, then
    \[
    \ip*{\rchi_{q, 1}}{A^* (\rchi_{q, 2} F_{q})}[L^2(M)] = 0 \ .
    \]
    Meanwhile, since $\rchi_{q, 2} \equiv 1$ on $\supp \rchi_{q, 1}$, then
    \[
    \ip*{\rchi_{q, 2} F_{q}}{A^*\rchi_{q,1}}[L^2(M)] = \ip*{F_{q}}{\Lap \rchi_{q,1}}[L^2(M)] = \rchi_{q,1} (q) = 1 \ .
    \]
    Thus, we have shown that
    \[
    \omega(u,v) = \sum_{q \in Q} (\conj{a_{q, 1}} b_{q, 2} - \conj{a_{q, 2}} b_{q, 1}) \ . \qedhere
    \]
\end{proof}
This Lemma shows that $(\dom(A^*)/\dom(A), \omega)$ and $(\C[N] \times \C[N], \Omega)$, where
\[
\Omega((a_{1}, a_{2}), (b_{1}, b_{2})) = \ip{a_{1}}{b_{2}}[\C[N]] - \ip{a_{2}}{b_{1}}[\C[N]] \ ,
\]
are symplectomorphic under the linear map $T \colon \dom(A^*)/\dom(A) \to \C[N] \times \C[N]$,
\begin{equation} \label{Eq: sympl coord of dom(A^*)/dom(A)}
    T \bigg( \sum_{q \in Q} a_{q, 1} \rchi_{q, 1} + a_{q, 2} \rchi_{q, 2} F_{q} + \dom(A) \bigg) \coloneqq \big( (a_{q, 1})_{q \in Q}, (a_{q, 2})_{q \in Q} \big) \ .
\end{equation}
Therefore, the self-adjoint extensions of $A$ are parametrized by the Lagrangian-Grassmanian of the complex symplectic vector space of dimension $2N$.

For any Lagrangian subspace $L$ of $(\dom(A^*) / \dom(A), \omega)$, define $\Lap_{Q, L}$ as the restriction of $A^*$ to the subspace
\begin{equation}
    \dom(\Lap_{Q, L}) \coloneqq \pi^{-1}(L) \ .
\end{equation}
These operators describe all point-perturbations on $Q$ of $\Lap$.

\subsection{Spectrum of \texorpdfstring{$A^*$}{A*}}
\label{subSec: Point-pert of Lap. Spectrum of A*}

For every $\lambda \in \R$, define $Z_{\lambda}^{q}$, $q \in M$, as the unique function in $\ker(\sqLap - \lambda)$ such that
\begin{equation} \label{Eq: def Zlambda}
    u(q) = \ip{Z_{\lambda}^{q}}{u}[L^2(M)] \qquad \forall \ u \in \ker(\sqLap - \lambda) \ .
\end{equation}
Existence and uniqueness of $Z_{\lambda}^{q} \in \ker(\sqLap - \lambda)$ is granted by the fact $\dim(\ker(\sqLap - \lambda)) < \infty$. Observe that $Z_{\lambda}^{q}$ is not normalized; even it may be the zero function too.

For $\eta \in \C$ and $\beta = (\beta_{q})_{q \in Q} \in \C[N]$, such that $\sum_{q \in Q} \beta_{q} \delta_{q} = 0$ on $\ker(\Lap {}- \eta)$, define
\begin{equation} \label{Eq: def GetaQbeta}
    G_{\eta}^{Q, \beta} \coloneqq \sum_{ \substack{ \lambda \in \Spec(\sqLap) \\ \lambda^2 \neq \eta} } \frac{1}{\lambda^2 - \eta} \sum_{q \in Q} \beta_{q} Z_{\lambda}^{q} \ .
\end{equation}
Observe that $G_{\eta}^{Q, \beta}$ is a linear combination of Green's function on $q \in Q$ for $\Lap$ at energy $\conj{\eta}$,
\[
\ip{G_{\eta}^{Q, \beta}}{(\Lap {}- \conj{\eta}) u}[L^2] = \sum_{q \in Q} \beta_{q} u(q) \ , \qquad \forall \ u \in H^2(M) \ .
\]
Note that $G_{\eta}^{Q, \beta}$ is real-valued if $\eta \in \R$, as a consequence of the following result.

\begin{Lemma} \label{Lemma: Geta real-valued}
    For every $q\in M$ and every $\lambda \in \Spec(\sqLap)$, the function $Z_{\lambda}^{q}$ is real-valued.
\end{Lemma}
\begin{proof}
    First, we note that if $\Lap u = \lambda^2 u$, then $\Lap (\Re u) = \lambda^2 (\Re u)$ and $\Lap (\Im u) = \lambda^2 (\Im u)$ as well. Hence, $\Im Z_{\lambda}^{q_0} \in \ker(\sqLap - \lambda)$ and
    \[
    [\Im Z_{\lambda}^{q}] (q) = \ip{Z_{\lambda}^{q}}{\Im Z_{\lambda}^{q}}[L^2] = \ip{\Re Z_{\lambda}^{q}}{\Im Z_{\lambda}^{q}}[L^2] - i \norm{\Im Z_{\lambda}^{q}}[L^2]^2 \ .
    \]
    Since $\Im Z_{\lambda}^{q}$ and $\Re Z_{\lambda}^{q}$ are real-valued, we read from this that $\Im Z_{\lambda}^{q} = 0$.
\end{proof}

The following theorem characterizes the eigenspaces of $A^*$, that will later help us characterize the eigenspaces of $\Lap_{Q, L}$.

\begin{Prop} \label{Prop: structure of EF of A*}
    Let $\eta \in \C$. $u_{\eta} \in \ker(A^* - \eta)$ if and only
    \begin{equation} \label{Eq: A* eigenfunctions form}
        u_{\eta} = w_{\eta} + G_{\eta}^{Q, \beta} \ ,
    \end{equation}
    for some $w_{\eta} \in \ker(\Lap {}- \eta)$ and $\beta \in \C[N]$ such that $\sum_{q \in Q} \beta_{q} \delta_{q} = 0$ on $\ker(\Lap {}- \eta)$. This choice of $w_{\eta}$ and $\beta$ are unique since the two terms in the decomposition \eqref{Eq: A* eigenfunctions form} are orthogonal in $L^2(M)$.
\end{Prop}

\begin{proof}
    By definition, $A^* u_{\eta} = \eta u_{\eta}$ if and only if $\ip{u_{\eta}}{(\Lap {}- \conj{\eta}) v}[L^2(M)] = 0$ for all $v \in \CinfK(M \setminus Q)$. The map $v \mapsto \ip{u_{\eta}}{(\Lap {}- \conj{\eta}) v}[L^2(M)]$ is a continuous linear functional on $H^2(M)$ that vanishes on the whole subspace $\CinfK(M \setminus Q)$. The set of all $H^{-2}$-distributions supported in $Q$ is $D = \Span{ \delta_{q} \colon q \in Q}$ ($d = 2,3$). Therefore, we look for all $u_{\eta} \in L^2(M)$ such that, in a distributional sense, $(\Lap {}- \conj{\eta}) \conj{u_{\eta}} \in D$, where $\conj{u_{\eta}}$ is the complex conjugate function of $u_{\eta}$.

    In an eigenfunction expansion,
    \[
    \delta_{q} = \sum_{\lambda \in \Spec(\sqLap)} Z_{\lambda}^{q} \qquad \text{in $H^{-2}(M)$.}
    \]
    Therefore, $(\Lap {}- \conj{\eta}) \conj{u_{\eta}} \in D$ if and only if
    \[
    u_{\eta} = w_{\eta} + G_{\eta}^{Q, \beta} \ ,
    \]
    for some $w_{\eta} \in \ker(\Lap {}- \eta)$, and complex numbers $\beta_{q}$ such that $\sum_{q \in Q} \beta_{q} \delta_{q} = 0$ on $\ker(\Lap {}- \eta)$. Here we used that $Z_{\lambda}^{q}$ are real-valued.
\end{proof}

\subsection{Description of point-perturbations of \texorpdfstring{$\Lap$}{Laplacian} through Hermitian matrices}
\label{subSec: Point-pert of Lap. L and matA}

For fixed and finite $Q \subseteq M$, we prove a connection between the two descriptions of point-perturbation of $\Lap$: $\Lap_{Q, L}$ for a Lagrangian subspace $L$ of $(\dom(A^*)/\dom(A), \omega)$ \cite{Hillairet2010}, and $(\Lap {}+ \sum_{q,p \in Q} \ket{\delta_{p}} \matA_{q,p} \bra{\delta_{q}})$ for an Hermitian matrix $\matA$ \cite{Yesha2018}.

\begin{Remark}
    Note is that the Lagrangian-Grassmanian of $\C[N] \times \C[N]$ is compact, but the space of $N \times N$ complex Hermitian matrices is not; thus this connection cannot be perfect. In some sense, this Lagrangian-Grassmanian is a compactification of this space of Hermitian matrices, in the same way as $\mathbb{CP}^{N}$ is a compactification of $\C[N]$. 
\end{Remark}

We start this section proving a fundamental result that connect the resolvents of $\Lap$ and $\Lap_{Q, L}$. Recall that for bounded operators $A$ and $B$ on a Hilbert space, one has the following resolvent identity for $\eta \in \C \setminus ( \Spec(A) \cup \Spec(B) )$,
\[
(A - \eta)^{-1} = (B - \eta)^{-1} + (A - \eta)^{-1} (B - A) (B - \eta)^{-1} \ .
\]

\begin{Theorem} \label{Thm: resolvent identity Lap LapL}
    Let $\eta \in \C \setminus \Spec(\Lap)$ and $L$ be a Lagrangian subspace of $(\dom(A^*) / \dom(A), \omega)$. Assume that $\eta \notin \Spec(\Lap_{Q, L})$. There exists an $N \times N$ complex matrix $\matA(L, \eta)$ such that
    \begin{equation} \label{Eq: resolvent identity}
        (\Lap_{Q, L} {}- \eta)^{-1} = (\Lap {} - \eta)^{-1} + \sum_{q, p \in Q} \ket*{G_{\eta}^{q}} [\matA(L, \eta)]_{p, q} \bra[\big]{G_{\conj{\eta}}^{p}}
    \end{equation}
    as operators on $L^2(M)$. Moreover, if $\eta \in \R$, then $\matA(L, \eta)$ is Hermitian.
\end{Theorem}

\begin{proof}
    Identity \eqref{Eq: resolvent identity} on $L^2(M)$ will be a consequence of
    \begin{equation} \label{Eq: resolvent identity aux}
        (\Lap_{Q, L} {}- \eta)^{-1} (\Lap {}- \eta) = \Id_{\dom(\Lap)} + \sum_{q, p \in Q} \ket*{G_{\eta}^{q}} [\matA(L, \eta)]_{p, q} \, \delta_{p} \qquad \text{on $\dom(\Lap)$} \ ,
    \end{equation}
    and the property $u(p) = \ip{G_{\conj{\eta}}^{p} }{ (\Lap {}- \eta) u}[L^2(M)]$ for $\eta \in \C \setminus \Spec(\Lap)$, $u \in \dom(\Lap)$. We will construct the matrix $\matA(L, \eta)$ through several lemma. First and foremost, we need a matrix description of Lagrangian subspaces $L$ and for the condition $\eta \notin \Spec(\Lap_{Q, L})$.

    Since $(\dom(A^*)/\dom(A), \omega)$ is symplectomorphic to $(\C[N] \times \C[N], \Omega)$ through the coordinate map $T$ defined in \eqref{Eq: sympl coord of dom(A^*)/dom(A)}, we may work in this coordinates. Using the points of $Q$ as indices, we may identify
    \[
    T(\rchi_{q, 1} + \dom(A) ) = e_{q} \ , \qquad T(\rchi_{q, 2} F_{q} + \dom(A)) = f_{q} \ ,
    \]
    where $\{e_{q}, f_{q} \colon q \in Q \}$ denotes the natural symplectic basis of $\C[2N]$, $\Omega(e_{q}, f_{p}) = \delta_{q}^{p}$. For instance, the Lagrangian subspace $E \coloneqq \Span{e_{q} \colon q \in Q}$ would give rise to the trivial self-adjoint extension, $\Lap_E = \Lap$, because $\dom(\Lap_E) = \pi^{-1}(E) = H^2(M)$.

    Let $L \subseteq \C[N] \times \C[N]$ be a $N$-dimensional subspace, and $(C,S)$ a pair of $N \times N$ complex matrices such that
    \begin{equation} \label{Eq: L in terms of (C,S)}
        L = \begin{bmatrix} C \\ S \end{bmatrix} E \ .
    \end{equation}
    
    \begin{Lemma} \label{Lemma: L Lagrangian and matrices CS}
        The following holds: $L$ is Lagrangian if and only if the pair $(C,S)$ satisfies
        \begin{equation}\label{Eq: Prop of (C,S)}
            \begin{dcases} C^*C + S^*S = \Id_N \ , \\ C^*S = S^*C \ . \end{dcases}
        \end{equation}
        Note that $(C,S)$ satisfying \eqref{Eq: Prop of (C,S)} is equivalent to $(C + i S)$ being a unitary matrix.
    \end{Lemma}
    
    \begin{Lemma} \label{Lemma: matrices CS up to U unitary}
        Let two pairs of $N \times N$ matrices $(C_1, S_1)$ and $(C_2, S_2)$ satisfying \eqref{Eq: Prop of (C,S)}, and let $L_1$ and $L_2$ be the Lagrangian subspaces defined by \eqref{Eq: L in terms of (C,S)} respectively. The following holds
        \[
        L_1 = L_2 \qquad \iff \qquad \exists \ U \ \text{unitary} \quad \text{s.t.} \quad \begin{bmatrix} C_1 \\ S_1 \end{bmatrix} = \begin{bmatrix} C_2 U \\ S_2 U \end{bmatrix} \ .
        \]
    \end{Lemma}
    \begin{proof}
        If $C_1 = C_2 U$ and $S_1 = S_2 U$ for a certain unitary matrix $U$, it is immediate that $L_1 = L_2$. On the other hand, if $L_1 = L_2$, then there exists a change of basis matrix $B$ such that $C_1 = C_2 B$ and $S_1 = S_2 B$. Using that $(C_1, S_1)$ and $(C_2, S_2)$ satisfy \eqref{Eq: Prop of (C,S)}, we see
        \[
        \Id_{N} = (C_1)^* C_1 + (S_1)^*S_1 = B^* (C_2)^* C_2 B + B^* (S_2)^* S_2 B = B^* B \ ,
        \]
        then $B^{*} = B^{-1}$ and thus $B$ is unitary.
    \end{proof}

    Having a matrix description for the Lagrangian subspace $L$, we would like to translate the condition $\eta \notin \Spec(\Lap_{Q, L})$ into this language. Define the subspace of $\dom(A^*) / \dom(A)$ \ ,
    \begin{equation}
        \mathcal{G}_{\eta} \coloneqq \Span{ G_{\eta}^{q} + \dom(A) \colon q \in Q} \ .
    \end{equation}
    
    Thanks to \zcref{Prop: structure of EF of A*} and the fact that $\Lap_{Q, L} \coloneqq A^*|_{\dom(\Lap_{Q, L})}$, we obtain the following lemma.

    \begin{Lemma} \label{Lemma: eta not ev iff LcapGeta empty}
        Let $\eta \in \C \setminus \Spec(\Lap)$. $\eta$ is not an eigenvalue of $\Lap_{Q, L}$ if and only if $L \cap \mathcal{G}_{\eta} = \varnothing$.
    \end{Lemma}

    In addition, we have the following matrix description of $\mathcal{G}_{\eta}$.
    \begin{Lemma}
        For every $\eta \in \C \setminus \Spec(\Lap)$, $G_{\eta}^q - \rchi_{q,2} F_q \in \dom(\Lap)$. In addition, the subspace $\mathcal{G}_{\eta} \subseteq \dom(A^*)/\dom(A)$ is given by
        \begin{equation} \label{Eq: matrices for calG}
            \mathcal{G}_{\eta} = \begin{bmatrix} \mathbb{G}_{\eta} \\ \Id_N \end{bmatrix} E \qquad \text{for the matrix} \quad \mathbb{G}_{\eta} \coloneqq [ (G_{\eta}^{q} - \rchi_{q, 2} F_{q})(p) ]_{q,p \in Q} \ .
        \end{equation}
    \end{Lemma}
    
    \begin{proof}
        We prove that $G_{\eta}^q - \rchi_{q,2} F_q \in \dom(\Lap^*) = \dom(\Lap)$. Precisely, for $u \in \Cinf(M)$ (which is dense in $\dom(\Lap)$ for the graph-norm)
        \[
        \ip{G_{\eta}^{q} - \rchi_{q,2} F_{q}}{\Lap u}[L^2(M)] = \ip{G_{\eta}^{q}}{ (\Lap {}- \eta) u}[L^2(M)] + \ip{\conj{\eta} G_{\eta}^{q}}{u}[L^2(M)] - \ip{ F_{q}}{ \rchi_{q, 2} \Lap u}[L^2(M)] \ .
        \]
        The definition of $G_{\eta}^{q}$ for $\eta \in \C \setminus \Spec(\Lap)$ implies
        \[
        \ip{G_{\eta}^{q}}{ (\Lap {}- \eta) u}[L^2(M)] = u(q) \ .
        \]
        Meanwhile, using that $\rchi_{q, 2} \Lap u = \Lap(\rchi_{q, 2} u) + 2 \ip{\nabla u}{\nabla \rchi_{q, 2}}[g] - u \Lap \rchi_{q, 2}$ and arguing as we did in \zcref{Lemma: symplectic form on dom(A*)}, we find that (c.f. \eqref{Eq: aux pag6})
        \[
        \ip{ F_{q}}{ \rchi_{q, 2} \Lap u}[L^2(M)] = u(q) + \ip{ - \mathop{\mathrm{div}}( 2F_q \nabla \rchi_{q,2} ) - F_{q} \Lap \rchi_{q, 2} }{u}[L^2(M)] \ .
        \]
        Putting everything together, we have shown
        \[
        \ip{G_{\eta}^{q} - \rchi_{q,2} F_{q}}{\Lap u}[L^2(M)] = \ip{ \conj{\eta} G_{\eta}^{q} + \mathop{\mathrm{div}}( 2F_q \nabla \rchi_{q,2} ) + F_{q} \Lap \rchi_{q, 2} }{u}[L^2(M)] \qquad \forall \, u \in \Cinf(M) \ .
        \]
        Observe that the left term on the right-hand side is a $L^2(M)$. This means that $G_{\eta}^{q} - \rchi_{q,2} F_{q} \in \dom(\Lap^*) = \dom(\Lap)$ and
        \[
        \Lap( G_{\eta}^{q} - \rchi_{q,2} F_{q} ) = \conj{\eta} G_{\eta}^{q} - \mathop{\mathrm{div}}( 2F_q \nabla \rchi_{q,2} ) - F_{q} \Lap \rchi_{q, 2} \ .
        \]
        Therefore, we may write
        \begin{equation} \label{Eq: D(A*) coordinates for Getaq}
            G_{\eta}^{q} + \dom(A) = \sum_{p \in Q} (G_{\eta}^{q} - \rchi_{q, 2} F_{q})(p) \rchi_{p, 1} + \rchi_{p, 2} F_{p} + \dom(A) \ .
        \end{equation}
        as elements in $\dom(A^*) / \dom(A)$.
    \end{proof}

    \begin{Lemma} \label{Lemma: eta not ev iff C-GLetaS invertible}
        Let $\eta \in \C \setminus \Spec(\Lap)$, $L$ be a Lagrangian subspace of $(\dom(A^*) / \dom(A), \omega)$, and let $(C,S)$ be any pair of complex matrices given by \zcref{Lemma: L Lagrangian and matrices CS}. The following holds: $\eta$ is not an eigenvalue of $\Lap_{Q, L}$ if and only if $C - \mathbb{G}_{\eta} S$ is invertible.
    \end{Lemma}
    \begin{proof}
        Thanks to \zcref{Lemma: eta not ev iff LcapGeta empty}, we know that $\mathcal{G}_{\eta} \cap L = \{ 0 \}$ if and only if $\eta$ is not an eigenvalue of $\Lap_{Q, L}$. Using the matrix expressions for $L$ \eqref{Eq: L in terms of (C,S)} and $\mathcal{G}_{\eta}$ \eqref{Eq: matrices for calG}, we see that $\mathcal{G}_{\eta} \cap L = \{ 0 \}$ if and only if the $2N \times 2N$ matrix
        \[
        \begin{bmatrix} C & \mathbb{G}_{\eta} \\ S & \Id_N \end{bmatrix}
        \]
        is invertible. From the Weinstein-Aronszajn identity, we see that this matrix is invertible if and only if $C - \mathbb{G}_{\eta} S$ is invertible (Schur complement of $\Id_N$).
    \end{proof}

    \begin{Remark}
        We will later show that $\Lap_{Q, L}$ has compact resolvent using \zcref{Thm: resolvent identity Lap LapL}, hence, the equivalences in \zcref{Lemma: eta not ev iff LcapGeta empty} and \zcref{Lemma: eta not ev iff C-GLetaS invertible} can be restated in terms of $\eta \notin \Spec(\Lap_{Q, L})$.
    \end{Remark}

    Finally, we are in position to construct the matrix $\matA(L, \eta)$. Recall that $L$ is a fixed Lagrangian subspace and $\eta \in \C \setminus \Spec(\Lap)$ is such that $\eta \notin \Spec(\Lap_{Q, L})$. Let $(C, S)$ be a pair of $N \times N$ complex matrices given by \zcref{Lemma: L Lagrangian and matrices CS}, and define the $N \times N$ complex matrix
    \begin{equation} \label{Eq: def A(L,eta)}
        \matA(L, \eta) \coloneqq S (C - \mathbb{G}_{\eta} S)^{-1} \ .
    \end{equation}
    The following lemma proves that $\mathbb{A}(L, \eta)$ is independent of the choice of matrices $C$ and $S$.
    
    \begin{Lemma} \label{Lemma: prop of mat A(L, eta)}
        Let $\eta \in \C \setminus \Spec(\Lap)$ and $L$ a Lagrangian subspace of $(\dom(A^*) / \dom(A), \omega)$ such that $\eta \notin \Spec(\Lap_{Q, L})$. The following properties are true:
        \begin{enumerate}
            \item The matrix $\matA(L, \eta)$ is independent of the choice of $(C, S)$
            \item $\mathbb{G}_{\eta}$ is a symmetric matrix. 
            \item If $\eta \in \R$, then $\mathbb{G}_{\eta}$ is real. 
            \item If $\eta \in \R$, then $\mathbb{A}(L, \eta)$ is Hermitian.
        \end{enumerate}
    \end{Lemma}
    
    \begin{proof}
        (1) Let $(C', S')$ be another pair of complex matrices satisfying \zcref{Lemma: L Lagrangian and matrices CS}. Thanks to \zcref{Lemma: matrices CS up to U unitary}, there exists a $N \times N$ unitary matrix $U$ such that $C = C' U$ and $S = S' U$. Therefore,
        \[
        C (C - \mathbb{G}_{\eta} S)^{-1} = C'U (C'U - \mathbb{G}_{\eta} S'U)^{-1} = C' (C' - \mathbb{G}_{\eta} S')^{-1} \ .
        \]
        
        (2) For $q, p \in Q$, $p \neq q$,
        \[
        (\mathbb{G}_{\eta})_{q,p} = [G_{\eta}^{q} - \rchi_{q,2} F_q] (p) = G_{\eta}^{q}(p) = \sum_{ \substack{ \lambda \in \Spec(\sqLap) \\ \lambda^2 \neq \eta} } \frac{1}{\lambda^2 - \eta} Z_{\lambda}^{q}(p) \ .
        \]
        Since for all $\lambda \in \Spec(\sqLap)$ and all $q, p \in M$,
        \[
        Z_{\lambda}^{q}(p) = \ip{Z_{\lambda}^{p}}{Z_{\lambda}^{q}}[L^2(M)] = \ip{Z_{\lambda}^{q}}{Z_{\lambda}^{p}}[L^2(M)] = Z_{\lambda}^{p}(q)
        \]
        due to $Z_{\lambda}^{q}$ being real-valued, we have $(\mathbb{G}_{\eta})_{q,p} = (\mathbb{G}_{\eta})_{p,q}$, thus $\mathbb{G}_{\eta}$ is symmetric.
        
        (3) If we assume that $\eta \in \R$, we see that $\mathbb{G}_{\eta}$ is real-valued because $G_{\eta}^{q}$ and $\rchi_q F_q$ are both real-valued. 
        
        (4) If $\eta \in \R$, we know that $\mathbb{G}_{\eta}^* = \mathbb{G}_{\eta}$. We see that $\mathbb{A}(L, \eta)$ is Hermitian if and only if
        \[
        S (C - \mathbb{G}_{\eta} S)^{-1} = \big[ S (C - \mathbb{G}_{\eta} S)^{-1} \big]^* = (C^* - S^* \mathbb{G}_{\eta} )^{-1} S^* \ .
        \]
        This is equivalent to
        \[
        S = (C^* - S^* \mathbb{G}_{\eta})^{-1} S^* (C - \mathbb{G}_{\eta} S) \ ,
        \]
        which holds because
        \[
        S^* (C - \mathbb{G}_{\eta} S) = S^*C - S^* \mathbb{G}_{\eta} S = C^* S - S^* \mathbb{G}_{\eta} S = (C^* - S^* \mathbb{G}_{\eta} ) S \ ,
        \]
        thanks to \eqref{Eq: Prop of (C,S)}.
    \end{proof}

    We finally in position to prove \eqref{Eq: resolvent identity aux}. Define $P_{L, \eta} \coloneqq (\Lap_{Q, L} {}- \eta)^{-1} (\Lap {}- \eta)$. We have that for all $u \in \dom(\Lap)$, $(\Lap_{Q, L} {}- \eta) P_{L, \eta} u = (\Lap {}- \eta) u$. Since $A^*$ is an extension of $\Lap$ and $\Lap_{Q, L}$, we find that $(A^* - \eta) [P_{L, \eta}u - u] = 0$. Thanks to \zcref{Prop: structure of EF of A*}, since $\eta \in \C \setminus \Spec(\Lap)$ there exist a unique $\beta = (\beta_{q})_{q \in Q} \in \C[N]$, such that
    \begin{equation} \label{Eq: Lemma 2.13 aux 2}
        P_{L, \eta} u - u = \sum_{q \in Q} \beta_{q} G_{\eta}^{q}\ .
    \end{equation}
    We can calculate $\beta$ thanks to the fact that $P_{L, \eta} u$ belongs to $\dom(\Lap_{Q, L})$, and the help of the matrices we defined.

    We observe that as elements in $\dom(A^*)/\dom(A)$, we may write (see \eqref{Eq: D(A*) coordinates for Getaq})
    \begin{align*}
        P_{L, \eta} u + \dom(A) & = \sum_{p \in Q} \Big( u + \sum_{q \in Q} \beta_{q} G_{\eta}^{q} \Big) \rchi_{p,1} + \dom(A) \ , \\
        \Big( u + \sum_{q \in Q} \beta_{q} G_{\eta}^{q} \Big) \rchi_{p,1} + \dom(A) & = \Big[ u(p) + \sum_{q \in Q} \beta_{q} (G_{\eta}^{q} - \rchi_{q,2} F_{q}) (p) \Big] \rchi_{p,1} + \big[ \beta_p \big] \rchi_{p,2} F_p + \dom(A) \ .
    \end{align*}
    We may use this identity to work in coordinates. Let $(C,S)$ be a pair of complex matrices given by \zcref{Lemma: L Lagrangian and matrices CS}. Since $P_{L, \eta} u$ belongs to $D(\Lap_{Q, L})$, there exist $\alpha \in E \subseteq \C[2N]$ (see \eqref{Eq: L in terms of (C,S)}) such that, with an slight abuse of notation,
    \[
    P_{L, \eta} u = \begin{bmatrix} C \\ S \end{bmatrix} \vec{\alpha} \ ,
    \]
    which in coordinates, $\vec{u} \coloneqq (u(p))_{p \in Q}$, $\vec{\alpha} \coloneq (\alpha_{p})_{p \in Q}$, $\vec{\beta} = (\beta_{q})_{q \in Q} \in \C[N]$, translates into the system of linear equations
    \[
    \vec{u} + \mathbb{G}_{\eta} \vec{\beta} = C \vec{\alpha} \ , \qquad \vec{\beta} = S \vec{\alpha} \ .
    \]
    From this, we read that $\vec{u} = (C - \mathbb{G}_{\eta} S) \vec{\alpha}$, thus $\vec{\alpha} = (C - \mathbb{G}_{\eta} S)^{-1} \vec{u}$ thanks to \zcref{Lemma: eta not ev iff C-GLetaS invertible}. Therefore, the vector $\vec{\beta}$ is uniquely determined by
    \begin{equation} \label{Eq: value of vecbeta}
        \vec{\beta} = S(C - \mathbb{G}_{\eta}S)^{-1} \vec{u} = \matA(L, \eta) \vec{u} \ .
    \end{equation}
    Substituting \eqref{Eq: value of vecbeta} into \eqref{Eq: Lemma 2.13 aux 2}, we get
    \[
    P_{L, \eta} u = u + \sum_{q \in Q} \Big[ \matA(L, \eta)_{q, p} u(p) \Big] G_{\eta}^{q} \ , \qquad \forall \ u \in \dom(\Lap) \ .
    \]
    In operator terms, this is \eqref{Eq: resolvent identity}:
    \[
    (\Lap_{Q, L} {}- \eta)^{-1} (\Lap {}- \eta) = \Id_{\dom(\Lap)} + \sum_{q,p \in Q} G_{\eta}^{q} \big[ \matA(L, \eta) \big]_{q, p} \delta_{p} \ . \qedhere
    \]
\end{proof}

We wind up this section showing how the symplectic description $\Lap_{Q, L}$ and the matrix description $(\Lap {}+ \delta_{Q} \matA \delta_{Q})$ may be related one to the other. We define the $(H^{-2}, H^2)$ pairing for $f \in H^{-2}(M)$ and $u \in H^2(M)$ as follows
\[
\ip{f}{u}[H^{-2} \times H^2] = \ip{ (\Lap {}+ 1)^{-1} f}{ (\Lap {}+ 1) u}[L^2(M)] \ .
\]
Recall that $\ip{u}{v}[L^2(M)] = \int_{M} \conj{u(x)} v(x) \D{x}$.

\begin{Corollary} \label{Cor: connection between LapL and Lap+A}
    Let $\eta \in \R \setminus \Spec(\Lap)$ and a Lagrangian subspace $L$ of $(\dom(A^*)/\dom(A), \omega)$ such that $\eta \notin \Spec(\Lap_{Q, L})$. Let $\matA(L, \eta)$ given by \zcref{Thm: resolvent identity Lap LapL} and let $P_{L, \eta} \coloneqq (\Lap_{Q, L} {}- \eta)^{-1} (\Lap {}- \eta)$. For all $u \in \dom(\Lap)$,
    \begin{equation} \label{Eq: exp energy of LapL}
        \ip*{P_{L, \eta} u}{(\Lap_{Q, L} {}- \eta) P_{L, \eta} u}[L^2(M)] = \ip*{ (\Lap {}+ \delta_{Q} \, \matA(L, \eta) \, \delta_{Q} - \eta) u}{ u }[H^{-2} \times H^2] \ ,
    \end{equation}
    where $\delta_{Q} \matA \delta_{Q} = \sum_{q, p \in Q} \delta_{q} \matA_{q, p} \delta_{p} \colon H^2 \to H^{-2}$ is a $N$-rank mapping.
\end{Corollary}

\begin{proof}
    This is just a computation using both expressions for $P_{L, \eta}$ and the definition of $G_{\eta}^{q}$:
    \begin{align*}
        \ip*{P_{L, \eta} u}{(\Lap_{Q, L} {}- \eta) P_{L, \eta} u}[L^2(M)] & = \ip{ u + \sum_{q, p \in Q} [\matA(L, \eta)]_{q, p} u(p) G_{\eta}^{q} }{(\Lap {}- \eta) u}[L^2(M)] \\
        & = \ip{(\Lap {}- \eta) u}{u}[L^2(M)] + \sum_{q, p \in Q} \conj{ [\matA(L, \eta)]_{q, p} u(p) } u(q) \\
        & = \ip*{ (\Lap {}+ \delta_{Q} \, \matA(L, \eta) \, \delta_{Q} - \eta) u }{u}[H^{-2} \times H^{-2}]
    \end{align*}
    In the last line we used that $(\Lap {}- \eta)$ is self-adjoint and that $\matA(L, \eta)$ is Hermitian.
\end{proof}

\subsection{Spectrum of \texorpdfstring{$\Lap_{Q, L}$}{Lap L}}
\label{subSec: Point-pert of Lap. Spectrum of LapL}

\begin{Theorem} \label{Thm: structure of eigenfunction of LapL}
    $\Lap_{Q, L}$ has compact resolvent for every Lagrangian subspace $L$ of $(\dom(A^*) / \dom(A), \omega)$. In addition, given any pair of matrices $(C, S)$ such that \eqref{Eq: L in terms of (C,S)} holds,
    \begin{equation} \label{Eq: Spec(LapL) equation}
        \eta \in \Spec(\Lap_{Q, L}) \setminus \Spec(\Lap) \qquad \text{if and only if} \qquad \det(C - \mathbb{G}_{\eta} S) = 0 \ ,
    \end{equation}
    where $\mathbb{G}_{\eta}$ is given by \eqref{Eq: matrices for calG} (\emph{cf.} \cite[Equation 1.1]{KurlbergUeberschaer2014} and \cite[Section 2.2]{Yesha2018}).
    
    Moreover, if $u_{\eta} \in D(\Lap_{Q, L})$ is such that $\Lap_{Q, L} u_{\eta} = \eta u_{\eta}$ for some $\eta \in \Spec(\Lap_{Q, L})$, there exist $w_{\eta} \in \ker(\Lap {}- \eta)$ and $\beta = (\beta_{q})_{q \in Q} \in \C[N]$ such that
    \begin{equation} \label{Eq: LapL eigenfunctions form}
        u_{\eta} = w_{\eta} + G_{\eta}^{Q, \beta} \ , \qquad \sum_{q \in Q} \beta_{q} \delta_{q} = 0 \quad \text{on $\ker(\Lap {}- \eta)$.}
    \end{equation}
\end{Theorem}

\begin{proof}
    It is enough to prove that $(\Lap_{Q, L} {}- \eta)^{-1} \colon L^2(M) \to L^2(M)$ is compact for $\eta = \pm i$. Since $\pm i \in \C \setminus(\Spec(\Lap) \cup \Spec(\Lap_{Q, L}))$, we may invoke \zcref{Thm: resolvent identity Lap LapL} and write as operators in $L^2(M)$,
    \[
    (\Lap_{Q, L} {}- \eta)^{-1} = (\Lap {}- \eta)^{-1} + \sum_{q, p \in Q} \ket[\big]{G_{\eta}^{p}} [\matA(l, \eta)]_{p, q} \bra[\big]{G_{\eta}^{p}} \ .
    \]
    Therefore, we find that $(\Lap_{Q, L} {}- \eta)^{-1}$ is a compact operator in $L^2(M)$ because is the sum of a compact operator and a finite-rank operator.

    Equivalence \eqref{Eq: Spec(LapL) equation} is the content of \zcref{Lemma: eta not ev iff C-GLetaS invertible}, and \eqref{Eq: LapL eigenfunctions form} is a direct consequence of \eqref{Prop: structure of EF of A*}.
\end{proof}

\begin{Lemma} \label{Lemma: every GetaQbeta in dom(LapL)}
    Let $\eta \in \R$ and $\beta = (\beta_{q})_{q \in Q} \in \C[N] \setminus \{ 0\}$ such that $\sum_{q \in Q} \beta_{q} \delta_{q} = 0$ in $\ker(\Lap {}- \eta)$. There exists a self-adjoint extension $\Lap_{Q, L}$ of $A$ such that $G_{\eta}^{Q, \beta} \in \dom(\Lap_{Q, L})$, and thus $\Lap_{Q, L} G_{\eta}^{Q, \beta} = \eta G_{\eta}^{Q, \beta}$.
\end{Lemma}

\begin{proof}
    We are going to construct a $N-1$ dimensional subspace $S$ of $\dom(A^*)$ such that
    \[
    L \coloneqq \Big[ \Span[\C]{G_{\eta}^{Q, \beta}} + S\Big] + \dom(T)
    \]
    is a Langrangian subspace of $\dom(A^*)/\dom(A)$; \zcref{Thm: AL self-adjoint iff L Lagrangian} then gives $\Lap_{Q, L} \coloneqq A^*|_{\dom(\Lap_{Q, L})}$ for $\dom(\Lap_{Q, L}) \coloneqq \pi^{-1} L$ is a self-adjoint extension of $A$.

    Consider the family of bump functions $\{\rchi_{q, 1}\}_{q \in Q}$ associated to the family of pair-wise disjoint open neighborhoods of $q \in Q$, $\rchi_{q, 1} \in \CinfK(U_q)$ and $\rchi_{q, 1}(q) = 1$. One has the following: for every $q \in Q$,
    \begin{equation*}
        \omega(G_{\eta}^{Q, \beta}, \rchi_{q, 1}) = \ip{G_{\eta}^{Q, \beta}}{\Lap \rchi_{q, 1})}[L^2(M)] - \ip{\eta G_{\eta}^{Q, \beta}}{ \rchi_{q, 1}}[L^2(M)] = \ip{G_{\eta}^{Q, \beta}}{(\Lap {}- \eta) \rchi_{q, 1}}[L^2(M)] = \conj{\beta_{q}} \ .
    \end{equation*}
    For $\alpha = (\alpha_{q})_{q \in Q} \in \C[N]$, set $\rchi_{\alpha} = \sum_{q \in Q} \alpha_q \rchi_{q, 1}$, hence
    \begin{equation*}
        \omega(G_{\eta}^{Q, \beta}, \rchi_{\alpha}) = \sum_{q \in Q} \conj{\beta_q} \alpha_q \eqqcolon \ip{\beta}{\alpha}[\C] \ .
    \end{equation*}
    Define the following subspace of $\dom(A^*)$:
    \begin{equation} \label{Eq: Lemma GetaQbeta aux36}
        S \coloneq \{ \rchi_{\alpha} \colon \ip{\beta}{\alpha}[\C] = 0 \} \ .
    \end{equation}
    Since $\beta \neq 0$, $\dim(S) = N - 1$, therefore
    \begin{equation*}
        L \coloneqq \Big[ \Span[\C]{G_{\eta}^{Q, \beta}} + S\Big] + \dom(T)
    \end{equation*}
    is a $N$-dimensional subspace of $\dom(T^*)/\dom(T)$. Let us show that $L$ is isotropic, and thus Lagrangian. For that, we just need to prove that $\omega(u,v) = 0$ for all $u,v \in \Span[\C]{G_{\eta}^{Q, \beta}} + S$.
    \begin{enumerate}
        \item If $u,v \in S$, then $\omega(u, v) = 0$ because $\omega(\rchi_{q, 1}, \rchi_{p, 1}) = 0$ for all $q, p \in Q$.
        \item Then, $\omega(G_{\eta}^{Q, \beta}, \rchi_{\alpha}) = \ip{\beta}{\alpha}[\C[N]] = 0$ by our definition of $S$.
        \item Lastly, $\omega(G_{\eta}^{Q, \beta}, G_{\eta}^{Q, \beta}) = 0$ because $\eta \in \R$\footnote{Note that $\omega(u, u) = 0$ does not hold for every $u \in \dom(A^*)$; for instance, $\omega(G_{i}^{Q, \beta}, G_{i}^{Q, \beta}) = 2i \norm*{G_{i}^{Q, \beta}}[L^2(\S[2])]^2$.}.
    \end{enumerate}
\end{proof}

\begin{Remark}
    Observe that the self-adjoint perturbation constructed in \zcref{Lemma: every GetaQbeta in dom(LapL)} is a rank-one singular perturbation of $\Lap$. If one defines $B \coloneqq \Lap|_{\dom(B)}$ for $\dom(B) \coloneqq \dom(A) + S$ for the subspace $S$ defined in \eqref{Eq: Lemma GetaQbeta aux36}, then $\dom(B)$ is a closed one-codimensional subspace of $\dom(\Lap)$ given by
    \[
    \dom(B) = \Big\{ u \in \dom(\Lap) \colon \sum_{q \in Q} \conj{\beta_q} u(q) = 0 \Big\} \ .
    \]
    The operator $\Lap_{Q, L}$ in the proof is a self-adjoint extension of $B$, thus $\Lap_{Q, L}$ is a self-adjoint perturbation of $\Lap$ by the functional $\sum_{q \in Q} \beta_q \delta_{q}$.
\end{Remark}

\section{Approximation of new eigenfunctions by quasimodes}
\label{Sec: Approximation by quasimodes}

In this section we use different asymptotics of the spectral function of $\sqLap$, which hold under no additional assumptions or under some non-self-focality assumptions, to find quasimodes for linear combinations of Green's functions.

The point-wise spectral function of $\sqLap$ is
\[
E(q, p; X) \coloneqq \sum_{\lambda \leq X} Z_{\lambda}^{q}(p) \ , \qquad \text{ for $X \geq 0$.}
\]
Recall that $Z_{\lambda}^{q}(q) = \norm{Z_{\lambda}^{q}}[L^2(M)]^{2} \geq 0$. For $d \in \N$, let $\nu_{d} \coloneqq \frac{\vol(B^d)}{(2\pi)^{d}}$, where $\vol(B^d)$ is the volume of the Euclidean unit ball in $\R[d]$.

In \cite{Hormander1968}, diagonal and off-diagonal asymptotics of $E(q, p; X)$ were studied. One has the following: there exists $C > 0$ and $\Lambda > 1$ such that for every $q \in M$, $X \geq \Lambda$,
\begin{equation} \label{Eq: diag asymp spectral function}
    \abs*{ E(q, q; X) - \nu_{d} X^{d} } \leq C X^{d-1} \ ,
\end{equation}
and given a compact subset $K \subseteq \{ (q, p) \in M^2 \colon q \neq p \}$, there exists $C_K > 0$ and $\Lambda_{K} > 1$ such that for every $(q, p) \in K$, $X \geq \Lambda_{K}$ \ ,
\begin{align} \label{Eq: off-diag asymp spectral function}
    \abs*{ E(q,p; X) } \leq C_K X^{d-1} \ .
\end{align}

In \cite[Theorem 1.2]{SoggeZelditch2002}, \eqref{Eq: diag asymp spectral function} was improved under the condition that $q$ is non-self-focal \eqref{Eq: self-focal dir set q}: for every $\eps > 0$ there exists $\Lambda_{\eps} > 1$ such that for all $X \geq \Lambda_{\eps}$,
\begin{equation} \label{Eq: diag asymp spectral function improvement}
    \abs*{ E(q, q; X) - \nu_{d} X^{d} } \leq \eps X^{d-1} \ .
\end{equation}
In \cite[Theorem 3.3]{Safarov1989}, \eqref{Eq: off-diag asymp spectral function} was improved under the assumption that every pair $q,p \in K$ are mutually non-focal \eqref{Eq: mutually focal focal dir set qp} and either $q$ or $p$ is non-self-focal: for every $\eps > 0$ there exists $\Lambda_{K, \eps} > 1$ such that for every $q, p \in K$, $X \geq \Lambda_{K, \eps}$,
\begin{align} \label{Eq: off-diag asymp spectral function improvement}
    \abs*{ E(q,p; X) } \leq \eps X^{d-1} \ .
\end{align}

Let us introduce the semiclassical parameter $h > 0$, that will later tend to $0^+$. All these asymptotics on the spectral function $E(q, p; h^{-1})$ for small $h$ have consequences on the point-wise asymptotics on the family of quasimodes of $\sqLap$
\begin{equation} \label{Eq: quasimode Zhgammaq def}
    Z_{h, \gamma}^{q}(p) \coloneqq E(q, p; h^{-1} + \gamma) - E(q, p; h^{-1}) = \sum_{h^{-1} < \lambda \leq h^{-1} + \gamma} Z_{\lambda}^{q}(p)
\end{equation}
as $h \to 0^+$. Note that quasimodes $Z_{h, \gamma}^{q}$ enjoy the following property
\begin{equation} \label{Eq: zonal quasimodes ip}
    Z_{h, \gamma}^{q}(p) = \ip{Z_{h, \gamma}^{p}}{Z_{h, \gamma}^{q}}[L^2(M)] \qquad \forall \, q, p \in M \ .
\end{equation}

\begin{Lemma} \label{Lemma: quasimode point-wise asymp}
    Let $(M,g)$ be a closed Riemannian manifold of dimension $d \geq 2$ and fix some $q \in M$. There exist $C_0 > 0$ such that for every $\gamma > 0$ there exists $h_{\gamma} > 0$ such for every $0 < h < h_{\gamma}$
    \begin{equation} \label{Eq: diagonal quasimode asymp}
        \abs*{ Z_{h, \gamma}^{q} (q) - \nu_{d} d \gamma h^{1 - d} } \leq C_0 h^{1-d} \ .
    \end{equation}
    In addition, for every compact subset $K \subseteq M \setminus \{q\}$, there exist $C(K) > 0$ such that for every $\gamma > 0$ there exists $h_{K, \gamma} > 0$ such that for all $0 < h < h_{K, \gamma}$
    \begin{equation} \label{Eq: off-diagonal quasimode asymp}
        \sup_{p \in K} \abs*{Z_{h, \gamma}^{q} (p) } \leq C(K) h^{1 - d} \ .
    \end{equation}
\end{Lemma}

\begin{proof}
    Let $C > 0$ and $\Lambda > 0$ such that \eqref{Eq: diag asymp spectral function} holds. We have that for all $0 < h < \Lambda^{-1}$,
    \[
    \abs*{ E(q, q; h^{-1}) - \nu_{d} h^{-d}} \leq C h^{1-d} \ .
    \]
    Moreover, there exists some $C_1 > 0$ such that for every $\gamma > 0$ and every $0 < h < \gamma^{-\frac{d}{d-1}}$
    \[
    \abs*{ (h^{-1} + \gamma)^{-d} - h^{-d} - d \gamma h^{1-d} } \leq C_1 h^{1-d} \ .
    \]
    Therefore, for every $\gamma > 0$ and every $0 < h < \min \{ \Lambda^{-1}, \gamma^{-\frac{d}{d-1}} \} \eqqcolon h_{\gamma}$,
    \begin{align*}
        \abs*{ Z_{h, \gamma}^{q} (q) - \nu_{d} d \gamma h^{1 - d} } & \leq \abs*{ E(q, q; h^{-1} + \gamma) - \nu_{d} (h^{-1} + \gamma)^{d} } \\
        & \qquad + \abs*{ E(q, q; h^{-1}) - \nu_{d} (h^{-1})^d } \\
        & \qquad \qquad + \nu_{d} C_1 h^{1-d} \\
        & \leq \bigg[ 2^{d-1}C + C + \nu_{d} C_1 \bigg] h^{1-d}
    \end{align*}

    Let $K' \coloneqq \{q\} \times K$ be a compact subset of $M \times M$, and let $C_{K'} > 0$ and $\Lambda_{K'} > 1$ be such that \eqref{Eq: off-diag asymp spectral function} holds for all $(q, p) \in K'$. For every $\gamma > 0$ set $h_{K, \gamma} \coloneqq \min \{\Lambda_{K'}^{-1}, \gamma^{-1}\} > 0$. Then, for all $0 < h < h_{K, \gamma}$ and all $p \in K$,
    \begin{align*}
        \abs*{ Z_{h, \gamma}^{q} (p)} & \leq \abs*{E(q, p; h^{-1} + \gamma)} + \abs*{E(q, p; h^{-1})} \\
        & \leq C_{K'} (h^{-1} + \gamma)^{d-1} + C_{K'} h^{1 - d} \\
        & \leq [2^{d-1} C_{K'} + C_{K'} ] h^{1-d} \ . \qedhere
    \end{align*}
\end{proof}

\begin{Lemma} \label{Lemma: quasimode point-wise asymp improvement}
    Let $(M,g)$ be a closed Riemannian manifold of dimension $d \geq 2$ and fix some $q \in M$. Assume that $q$ is non-self-focal \eqref{Eq: self-focal dir set q}. For every $\gamma > 0$, $\eps > 0$, there exists $h_{\gamma, \eps} > 0$ such that for every $0 < h < h_{\gamma, \eps}$,
    \begin{equation} \label{Eq: diagonal quasimode asymp small gamma}
        \abs*{ Z_{h, \gamma}^{q} (q) - \nu_{d} d \gamma h^{1 - d} } \leq \eps h^{1-d} \ .
    \end{equation}
    In addition, for every compact subset $K \subseteq M \setminus \{q\}$ such that for every $p \in K$, $q,p$ are mutually non-focal \eqref{Eq: mutually focal focal dir set qp}, the following holds. For every $\gamma > 0$, $\eps > 0$, there exist $h_{\gamma, K, \eps} > 0$ such that for all $0 < h < h_{\gamma, K, \eps}$
    \begin{equation} \label{Eq: off-diagonal quasimode asymp small gamma}
        \sup_{p \in K} \abs*{Z_{h, \gamma}^{q} (p) } \leq \eps h^{1 - d} \ .
    \end{equation}
\end{Lemma}

\begin{proof}
    For each $\eps' > 0$ let $\Lambda_{\eps'} > 1$ such that \eqref{Eq: diag asymp spectral function improvement}. Fix $\eps > 0$ and set $\eps' = \frac{1}{2} (1 + 2^{d-1})^{-1} \eps > 0$. We get that for every $0 < h < (\Lambda_{\eps'})^{-1}$,
    \[
    \abs*{ E(q, q; h^{-1}) - \nu_{d} (h^{-1})^{d} } \leq \eps' h^{1-d} \ .
    \]
    Moreover, for every $\gamma > 0$ there exists some $h(\gamma, \eps) > 0$ such that for all $0 < h < h(\gamma, \eps)$,
    \[
    \abs*{ (h^{-1} + \gamma)^{-d} - h^{-d} - d \gamma h^{1-d} } \leq \frac{\eps}{2} h^{1-d} \ .
    \]
    Therefore, for every $\gamma > 0$, every $\eps > 0$, and every $0 < h < h_{\gamma, \eps} \coloneqq \min\{ (\Lambda_{\eps'})^{-1}, h(\gamma, \eps) \}$,
    \begin{align*}
        \abs*{ Z_{h, \gamma}^{q} (q) - \nu_{d} d \gamma h^{1 - d} } & \leq \abs*{ E(q, q; h^{-1} + \gamma) - \nu_{d} (h^{-1} + \gamma)^{d} } \\
        & \qquad + \abs*{ E(q, q; h^{-1}) - \nu_{d} (h^{-1})^d } \\
        & \qquad \qquad + \frac{\eps}{2} h^{1-d} \\
        & \leq \bigg[ 2^{d-1}\eps' + \eps' + \frac{\eps}{2} \bigg] h^{1-d} = \eps h^{1-d}
    \end{align*}

    Let $K' \coloneqq \{q\} \times K \subseteq M \times M$. Thanks to our assumptions on $K$, the compact theorem $K'$ satisfies the hypothesis of \cite[Theorem 3.3]{Safarov1989}, so estimate \eqref{Eq: off-diag asymp spectral function improvement} is at our disposal. Fix $\eps > 0$ and let $\eps' \coloneqq (1 + 2^{d-1})^{-1} \eps > 0$. Let $\Lambda_{K', \eps'} > 0$ such that \eqref{Eq: off-diag asymp spectral function improvement} holds. Then fix $\gamma > 0$ and set $h_{\gamma, K, \eps} \coloneqq \min\{ (\Lambda_{K', \eps'})^{-1}, \gamma^{-1} \} > 0$. For every $0 < h < h_{\gamma, K, \eps}$ and every $p \in K$,
    \begin{align*}
        \abs*{ Z_{h, \gamma}^{q} (p)} & \leq \abs*{E(q, p; h^{-1} + \gamma)} + \abs*{E(q, p; h^{-1})} \\
        & \leq \eps' (h^{-1} + \gamma)^{d-1} + \eps' h^{1 - d} \\
        & \leq \eps h^{1-d} \ . \qedhere
    \end{align*}
\end{proof}

\zcref{Lemma: quasimode point-wise asymp} and \zcref{Lemma: quasimode point-wise asymp improvement} in combination with \eqref{Eq: zonal quasimodes ip} roughly say that $Z_{h, \gamma}^{q}$ and $Z_{h, \gamma}^{p}$ should become orthogonal as either $\gamma \to \infty$ or $h \to 0^+$. \zcref{Lemma: gamma for density} and \zcref{Lemma: gamma for density improvement} below make this statement precise. For $\beta = (\beta_{q})_{q \in Q} \in \C[N]$, $h > 0$, $\gamma > 0$, define
\begin{equation} \label{Eq: ZhgammaQbeta def}
    Z_{h, \gamma}^{Q, \beta} \coloneqq \sum_{q \in Q} \beta_{q} Z_{h, \gamma}^{q} \ .
\end{equation}

\begin{Lemma} \label{Lemma: gamma for density}
    Let $(M, g)$ be a closed Riemannian manifold of dimension $d \geq 2$, and let $Q \subseteq M$ be a finite subset, $N \coloneqq \card Q$. There exists $C_{Q} > 0$ such that for every $\gamma > 0$ there exists $h_{Q, \gamma} > 0$ such that for all $\beta = (\beta_{q})_{q \in Q} \in \C[N]$, $0 < h < h_{Q, \gamma}$
    \begin{equation} \label{Eq: norm asymp of beta-zonal quasimode}
        \abs[\bigg]{ \norm*{Z_{h, \gamma}^{Q, \beta} }[L^2(M)]^2 - \nu_{d} d \gamma \ \norm{\beta}[\ell^2(Q)]^2 h^{1-d} } \leq C_{Q} \ \norm{\beta}[\ell^2(Q)]^2 h^{1-d} \ ,
    \end{equation}
    where for $k = 1, 2$,
    \[
    \norm{\beta}[\ell^k(Q)] \coloneqq \bigg[ \sum_{q \in Q} \abs{\beta_{q}}^k \bigg]^{\frac{1}{k}} \ .
    \]
\end{Lemma}

\begin{proof}
    Thanks to the identity \eqref{Eq: zonal quasimodes ip} we have
    \[
    \norm[\bigg]{\sum_{q \in Q} \beta_{q} Z_{h, \gamma}^{q} }[L^2(M)]^2 = \sum_{q \in Q} \abs{\beta_{q}}^2 Z_{h, \gamma}^{q} (q) + \sum_{\substack{q, p \in Q \\ p \neq q} } \Re(\conj{\beta_{p}} \beta_{q}) Z_{h, \gamma}^{q}(p) \ .
    \]
    On the one hand, thanks to \zcref{Lemma: quasimode point-wise asymp} there exists $C_{Q, 1} > 0$ such that for all $\gamma > 0$ there exists $h_{Q, \gamma} > 0$ such that for all $0 < h < h_{Q, \gamma, 1}$
    \[
    \abs*{Z_{h, \gamma}^{q}(p)} \leq C_{Q,1} h^{1-d} \qquad \forall \, p \neq q \ .
    \]
    Consequently, using that $\norm{\beta}[\ell^1(Q)] \leq \sqrt{N} \norm{\beta}[\ell^2(Q)]$,
    \[
    \abs[\bigg]{ \sum_{\substack{q, p \in Q \\ p \neq q} } \Re(\conj{\beta_{p}} \beta_{q}) Z_{h, \gamma}^{q}(p) } \leq \norm{\beta}[\ell^1(Q)]^2 \ C_{Q,1} h^{1 - d} \leq N C_{Q,1} \ \norm{\beta}[\ell^2(Q)]^2 \ h^{1-d} \ .
    \]
    On the other hand, thanks to \zcref{Lemma: quasimode point-wise asymp} too, there exists $C_{Q, 2} > 0$ such that for every $\gamma > 0$ there exists $h_{Q, \gamma, 2} > 0$ such that for all $0 < h < h_{Q, \gamma, 2}$ and all $q \in Q$,
    \[
    \abs*{Z_{h, \gamma}^{q}(q) - \nu_{d} d \gamma h^{1-d} } \leq C_{Q, 2} h^{1 - d} \ .
    \]
    Therefore, for all $0 < h < h_{Q, \gamma, 2}$,
    \[
    \abs[\bigg]{\sum_{q \in Q} \abs{\beta_{q}}^2 Z_{h, \gamma}^{q} (q) - \nu_{d} d \gamma \ \norm{\beta}[\ell^2(Q)]^2 h^{1-d} } \leq C_{Q, 2} \ \norm{\beta}[\ell^2(Q)]^2 h^{1-d} \ .
    \]
    Putting everything together, we find that for all $\gamma > 0$ and all $0 < h < h_{Q, \gamma} \coloneqq \min \{ h_{Q, \gamma, 1}, h_{Q, \gamma, 2}\}$,
    \[
    \abs[\bigg]{ \norm[\bigg]{\sum_{q \in Q} \beta_{q} Z_{h, \gamma}^{q} }[L^2(M)]^2 - \nu_{d} d \gamma \ \norm{\beta}[\ell^2(Q)]^2 h^{1-d} } \leq \Big[ N C_{Q, 1} + C_{Q, 2} \Big] \norm{\beta}[\ell^2(Q)]^2 h^{1-d} \ . \qedhere
    \]
\end{proof}

\begin{Lemma} \label{Lemma: gamma for density improvement}
    Let $(M, g)$ be a closed Riemannian manifold of dimension $d \geq 2$, and let $Q \subseteq M$ be a finite subset, $N \coloneqq \card Q$. Assume that $Q$ is such the following point-wise asymptotics hold.
    \begin{enumerate}
        \item For every $\gamma > 0$, $\eps > 0$ there exists $h_{Q, \gamma, \eps, 1} > 0$ such that \eqref{Eq: diagonal quasimode asymp small gamma} holds for all $0 < h < h_{Q, \gamma, \eps, 1}$ and all $q \in Q$.
        \item For every $\gamma > 0$, $\eps > 0$ there exists $h_{Q, \gamma, \eps, 2} > 0$ such that \eqref{Eq: off-diagonal quasimode asymp small gamma} holds for all $0 < h < h_{Q, \gamma, \eps, 2}$ and all $q, p \in Q$, $p \neq q$.
    \end{enumerate}
    Then, for every $\gamma > 0$ and every $\eps > 0$ there exists $h_{Q, \gamma, \eps} > 0$ such that for all $\beta = (\beta_{q})_{q \in Q} \in \C[N]$, $0 < h < h_{Q, \gamma, \eps}$
    \begin{equation} \label{Eq: norm asymp of beta-zonal quasimode improvement}
        \abs[\bigg]{ \norm*{Z_{h, \gamma}^{Q, \beta} }[L^2(M)]^2 - \nu_{d} d \gamma \ \norm{\beta}[\ell^2(Q)]^2 h^{1-d} } \leq \eps \ \norm{\beta}[\ell^2(Q)]^2 h^{1-d} \ ,
    \end{equation}
    where for $k = 1, 2$,
    \[
    \norm{\beta}[\ell^k(Q)] \coloneqq \bigg[ \sum_{q \in Q} \abs{\beta_{q}}^k \bigg]^{\frac{1}{k}} \ .
    \]
\end{Lemma}

\begin{Remark}
    Hypothesis (1) and (2) hold as soon as $Q$ is non-self-focal (see \eqref{Eq: self-focal dir set q}, \eqref{Eq: mutually focal focal dir set qp}) thanks to \zcref{Lemma: quasimode point-wise asymp improvement}.
\end{Remark}

\begin{proof}
    Fix $\gamma > 0$ and $\eps > 0$. Thanks to the identity \eqref{Eq: zonal quasimodes ip} we have
    \[
    \norm[\bigg]{\sum_{q \in Q} \beta_{q} Z_{h, \gamma}^{q} }[L^2(M)]^2 = \sum_{q \in Q} \abs{\beta_{q}}^2 Z_{h, \gamma}^{q} (q) + \sum_{\substack{q, p \in Q \\ p \neq q} } \Re(\conj{\beta_{p}} \beta_{q}) Z_{h, \gamma}^{q}(p) \ .
    \]
    On the one hand, because of (2) on the hypothesis there exists $h_{Q, \gamma, \eps, 1} > 0$ such that for all $0 < h < h_{Q, \gamma, \eps, 1}$
    \[
    \abs*{Z_{h, \gamma}^{q}(p)} \leq \frac{\eps}{2N} h^{1-d} \qquad \forall \, p \neq q \ .
    \]
    Consequently, using that $\norm{\beta}[\ell^1(Q)] \leq \sqrt{N} \norm{\beta}[\ell^2(Q)]$,
    \[
    \abs[\bigg]{ \sum_{\substack{q, p \in Q \\ p \neq q} } \Re(\conj{\beta_{p}} \beta_{q}) Z_{h, \gamma}^{q}(p) } \leq \norm{\beta}[\ell^1(Q)]^2 \ \frac{\eps}{2N} h^{1 - d} \leq \frac{\eps}{2} \ \norm{\beta}[\ell^2(Q)]^2 \ h^{1-d} \ .
    \]
    On the other hand, because of (1) on the hypothesis there exists $h_{Q, \gamma, \eps, 2} > 0$ such that for all $0 < h < h_{Q, \gamma, \eps, 2}$ and all $q \in Q$,
    \[
    \abs*{Z_{h, \gamma}^{q}(q) - \nu_{d} d \gamma h^{1-d} } \leq \frac{\eps}{2} h^{1 - d} \ .
    \]
    Therefore, for all $0 < h < h_{Q, \gamma, \eps, 2}$,
    \[
    \abs[\bigg]{\sum_{q \in Q} \abs{\beta_{q}}^2 Z_{h, \gamma}^{q} (q) - \nu_{d} d \gamma \ \norm{\beta}[\ell^2(Q)]^2 h^{1-d} } \leq \frac{\eps}{2} \ \norm{\beta}[\ell(Q)]^2 h^{1-d} \ .
    \]
    Putting everything together, we find that for all $0 < h < h_{Q, \gamma, \eps} \coloneqq \min \{ h_{Q, \gamma, \eps, 1}, h_{Q, \gamma, \eps, 2}\}$,
    \[
    \abs[\bigg]{ \norm[\bigg]{\sum_{q \in Q} \beta_{q} Z_{h, \gamma}^{q} }[L^2(M)]^2 - \nu_{d} d \gamma \ \norm{\beta}[\ell^2(Q)]^2 h^{1-d} } \leq \eps \norm{\beta}[\ell^2(Q)]^2 h^{1-d} \ . \qedhere
    \]
\end{proof}

We are finally in position to present and prove the theorem at the core of the paper. Since we are interested in the semiclassical regime, we introduce a small notational change that will last for the rest of the paper. We will no longer be interested in the functions $G_{\eta}^{Q, \beta}$ for arbitrary $\eta \in \R$, but only for positive $\eta$, hence for $h > 0$ and $\beta = (\beta_{q})_{q \in Q} \in \C[N]$ such that $\sum_{q \in Q} \conj{\beta_q} \delta_{q} \equiv 0$ on $\ker(\Lap {}- h^{-2})$, we set $G_{h}^{Q, \beta} \in L^2(M)$ as the unique function satisfying
\begin{equation*}
    \ip*{G_{h}^{Q, \beta}}{ (\Lap {}- h^{-2}) u}[L^2(M)] = \sum_{q \in Q} \conj{\beta_{q}} u(q) \qquad \forall \, u \in \dom(\Lap) \ .
\end{equation*}
One can prove that $G_{h}^{Q, \beta}$ has the following eigenfunction expansion
\begin{equation*} 
    G_{h}^{Q, \beta} \coloneqq \sum_{ \substack{ \lambda \in \Spec(\sqLap) \\ \lambda \neq h^{-1} } } \frac{1}{\lambda^2 - h^{-2}} Z_{\lambda}^{Q, \beta} \ .
\end{equation*}
In addition, define
\begin{equation*} 
    g_{h}^{Q, \beta} \coloneqq \frac{1}{\norm*{G_{h}^{Q, \beta}}[L^2(M)]} G_{h}^{Q, \beta} \ .
\end{equation*}
Observe that for $z \in \C \setminus \{ 0\}$, $g_{h}^{Q, z \beta} = \frac{z}{ \abs{z} } g_{h}^{Q, \beta}$.

Let $\Pi_I$ be the spectral projector of $\sqLap$ onto the interval $I \subseteq \R$, $\Pi_I = \charf_{I}(\sqLap)$. For any $r > 0$, $h > 0$, set $I_{h}(r) \coloneqq (h^{-1} - r,h^{-1} + r] \subseteq \R$.

\begin{Theorem} \label{Thm: quasimodes for GF}
    Let $(M, g)$ be a closed Riemannian manifold of dimension $d = 2$ or $d = 3$, and let $Q \subseteq M$ be a finite set of points, $N \coloneqq \card Q$. The following holds
    \begin{enumerate}
        \item There exists $C = C(Q) > 0$, $\gamma = \gamma(Q) > 0$, such that the following holds: for all $\Upsilon \in \N$ there exists $h_0 = h_{0}(Q, \Upsilon) > 0$ such that for all $\beta = (\beta_{q})_{q \in Q} \in \C[N]$, $0 < h < h_0$, $\sum_{q \in Q} \conj{\beta_q} \delta_{q} \equiv 0$ on $\ker(\Lap {}- h^{-2})$
        \begin{equation} \label{Eq: quasimode gh}
            \norm*{g_{h}^{Q, \beta} - \Pi_{I_{h}(\Upsilon \gamma)} g_{h}^{Q, \beta}}[L^2(M)] \leq C \Upsilon^{-1/2} \ .
        \end{equation}
        \item Assume that $Q$ is such the following point-wise asymptotics hold.
        \begin{enumerate}
            \item For every $\gamma > 0$, $\eps > 0$ there exists $h_{Q, \gamma, \eps, 1} > 0$ such that \eqref{Eq: diagonal quasimode asymp small gamma} holds for all $0 < h < h_{Q, \gamma, \eps, 1}$ and all $q \in Q$.
            \item For every $\gamma > 0$, $\eps > 0$ there exists $h_{Q, \gamma, \eps, 2} > 0$ such that \eqref{Eq: off-diagonal quasimode asymp small gamma} holds for all $0 < h < h_{Q, \gamma, \eps, 2}$ and all $q, p \in Q$, $p \neq q$.
        \end{enumerate}
        There exists $C = C(Q) > 0$ such that for every $\gamma > 0$ the following holds: there exists $h_0 = h_0(Q, \gamma) > 0$ such that for all $\beta = (\beta_{q})_{q \in Q} \in \C[N]$, $0 < h < h_0$, $\sum_{q \in Q} \conj{\beta_q} \delta_{q} \equiv 0$ on $\ker(\Lap {}- h^{-2})$,
        \begin{equation} \label{Eq: quasimode gh improvement}
            \norm*{g_{h}^{Q, \beta} - \Pi_{I_{h}(\gamma)} g_{h}^{Q, \beta}}[L^2(M)] \leq C \gamma^{1/2} \ .
        \end{equation}
    \end{enumerate}
\end{Theorem}

\subsection{Proof of \texorpdfstring{\zcref{Thm: quasimodes for GF}}{Theorem quasimodes for GF}}
\label{subSec: Approximation by quasimodes. Proof of Thm: quasimodes for GF} 

In every summation, the index $\lambda$ will always run through $\Spec(\sqLap) \setminus \{h^{-1}\}$, so we simply write $\sum_{\lambda}$ instead. Applying Pythagorean theorem to $G_{h}^{Q, \beta}$ we find that
\begin{equation} \label{Eq: Ghnorm eigenfun expansion}
    \norm*{G_{h}^{Q, \beta}}[L^2]^2 = \sum_{\lambda} \frac{1}{(\lambda^2 - h^{-2})^2} \norm[\bigg]{ \sum_{q \in Q} \beta_{q} Z_{\lambda}^{q}}[L^2]^2 \ .
\end{equation}
Understanding a series over the spectrum of an operator is much more difficult than understanding a series over $\N$ or over $\Z$. This is what motivates us to split the eigenfunction expansion of $G_{h}^{Q, \beta}$ into a countable collection of intervals of length $\gamma$; this is where quasimodes $Z_{h, \gamma}^{Q, \beta}$ come into play.

\subsubsection{Proof of \texorpdfstring{\zcref{Thm: quasimodes for GF}(1)}{Theorem quasimodes(1)}}
\label{subsubSec: Proof of Thm quasimodes for GF(1)}
Let $C_{Q} > 0$ be given by \zcref{Lemma: gamma for density}, choose
\[
\gamma = \gamma_{Q} \coloneqq \frac{C_{Q} + 1}{d \nu_{d}} \ ,
\]
and let $h_{Q, \gamma} > 0$ be given by \zcref{Lemma: gamma for density} so that
\[
\abs[\bigg]{ \norm*{Z_{h, \gamma}^{Q, \beta} }[L^2(M)]^2 - \nu_{d} d \gamma \ \norm{\beta}[\ell^2(Q)]^2 h^{1-d} } \leq C_{Q} \ \norm{\beta}[\ell^2(Q)]^2 h^{1-d} \qquad \forall \, 0 < h < h_{Q, \gamma} \ .
\]
Let $\Upsilon \in \N$, set $h_{0}(Q, \Upsilon) \coloneqq \min\{ h_{Q, \gamma}, \frac{1}{2} \gamma^{-1}, \Upsilon^{-1} \}$, and write
\[
G_{h}^{Q, \beta} = \Pi_{I_h (\Upsilon \gamma)} G_{h}^{Q, \beta} + [\Id - \Pi_{I_h (\Upsilon \gamma)}] G_{h}^{Q, \beta} \ .
\]
We will prove there exists $C = C(Q) > 0$ (independent of $\Upsilon$) such that for all $\beta = (\beta_{q})_{q \in Q} \in \C[N]$, $0 < h < h_0$,
\begin{align*} 
    \norm*{\Pi_{I_{h}(\Upsilon \gamma)} G_{h}^{Q, \beta} }[L^2(M)]^2 & \geq \frac{1}{9\gamma^2} \norm{\beta}[\ell^2(Q)]^2 h^{3-d} \ ,\\
    \norm*{G_{h_n}^{q} - \Pi_{I_{h}(\Upsilon \gamma)} G_{h}^{Q, \beta} }[L^2(M)]^2 & \leq C \norm{\beta}[\ell^2(Q)]^2 \frac{h^{3-d}}{\Upsilon} \ .
\end{align*}
From these estimates, \eqref{Eq: quasimode gh} follows.

A lower-bound for the $L^2$ norm of $\Pi_{I_h(\Upsilon \gamma)} G_{h}^{Q, \beta}$ is very simple to obtain thanks to \zcref{Lemma: gamma for density} and our choice of $\gamma$:
\begin{equation} \label{Eq: Thm large quasimode Gh aux 1}
\begin{aligned}
    \norm*{\Pi_{I_h(\Upsilon \gamma)} G_{h}^{Q, \beta} }[L^2(M)]^2 & \geq \sum_{h^{-1} < \lambda \leq h^{-1} + \gamma} \frac{1}{(\lambda^2 - h^{-2})^2} \norm[\bigg]{ \sum_{q \in Q} \beta_{q} Z_{\lambda}^{q} }[L^2(M)]^2 \\
    & \geq \frac{1}{\gamma^2} \frac{1}{(2h^{-1} + \gamma)^2} \norm*{Z_{h, \gamma}^{Q, \beta}}[L^2(M)]^2 \\
    & \geq \frac{1}{\gamma^2} \frac{1}{(3h^{-1})^2} \norm{\beta}[\ell^2(Q)]^2 h^{1-d} = \frac{1}{9\gamma^2} \norm{\beta}[\ell^2(Q)]^2 h^{3-d} \ .
\end{aligned}
\end{equation}

To find an upper bound of the $L^2$ norm of $[\Id - \Pi_{I_h(\Upsilon \gamma)}] G_{h}^{Q, \beta}$ we need to introduce a bit notation. Let $K_h \in \N$ be the smallest integer $k$ such that $k \gamma \geq h^{-1}$, and define the collection of intervals
\begin{align*}
    I_{h, k} & \coloneqq (h^{-1} + k \gamma, h^{-1} + (k +1) \gamma] \qquad k > -K_h \ , \\
    I_{h, K_h} & \coloneqq [0, \gamma] \ .
\end{align*}
We further split $[\Id - \Pi_{I_h(\Upsilon \gamma)}] G_{h}^{Q, \beta}$ into a lower-tail series and a upper-tail series:
\[
\norm*{ [\Id - \Pi_{I_h(\Upsilon \gamma)}] G_{h}^{Q, \beta} }[L^2(M)]^2 = \bigg[ \sum_{k = -K_h}^{-\Upsilon-1} + \sum_{k = \Upsilon}^{\infty} \bigg] \sum_{\lambda \in I_{h, k}} \frac{1}{(\lambda^2 - h^{-2})^2} \norm[\bigg]{\sum_{q \in Q} \beta_{q} Z_{\lambda}^{q}}[L^2(M)]^2 \ .
\]
For each $k > - K_h$ define $h_{k, \gamma}^* \coloneqq (h^{-1} + k \gamma)^{-1}$, so that $I_{h, k} = ( (h_{k, \gamma}^*)^{-1}, (h_{k, \gamma}^*)^{-1} + \gamma]$. On each interval $I_{h,k}$, $k \geq \Upsilon$, we may obtain the following upper bound
\begin{align*}
\sum_{\lambda \in I_{h, k}} \frac{1}{(\lambda^2 - h^{-2})^2} \norm[\bigg]{\sum_{q \in Q} \beta_{q} Z_{\lambda}^{q}}[L^2(M)]^2 & \leq \frac{1}{(k \gamma)^2} \frac{1}{(2h^{-1} + k \gamma)^2} \norm*{Z_{h_{k, \gamma}^*, \gamma}^{Q, \beta}}[L^2(M)]^2 \\
& \leq \frac{1}{k^2} \frac{C_Q}{\gamma^2} \norm{\beta}[\ell^2(Q)]^2 \frac{ (h^{-1} + k \gamma)^{d-1} }{(2h^{-1} + k \gamma)^2} \\
& \leq \frac{1}{k^2} \frac{C_Q}{\gamma^2} \norm{\beta}[\ell^2(Q)]^2 h^{3-d} \ ,
\end{align*}
where we used that
\[
\frac{ (h^{-1} + k \gamma)^{d-1} }{(2h^{-1} + k \gamma)^2} \leq (2h^{-1} + k \gamma)^{d-3} = h^{3-d} (2 + k \gamma h)^{d-3} \leq h^{3-d}
\]
because $k \geq 1$ and $1 \leq d \leq 3$. Therefore, we can conclude that
\begin{equation} \label{Eq: Thm large quasimode Gh aux 2}
\begin{aligned}
    \sum_{k = \Upsilon}^{\infty} \sum_{\lambda \in I_{h, k}} \frac{1}{(\lambda^2 - h^{-2})^2} \norm[\bigg]{\sum_{q \in Q} \beta_{q} Z_{\lambda}^{q}}[L^2(M)]^2 & \leq \frac{C_Q}{\gamma^2} \norm{\beta}[\ell^2(Q)]^2 h^{3-d} \sum_{k = \Upsilon}^{\infty} \frac{1}{k^2} \\
    & \leq \frac{2C_Q}{\gamma^2} \norm{\beta}[\ell^2(Q)]^2 \frac{h^{3-d}}{\Upsilon} \ .
\end{aligned}
\end{equation}
On the other hand, on each interval $I_{h, -k}$, $\Upsilon + 1 < k < K_h$,
\begin{align*}
\sum_{\lambda \in I_{h, -k}} \frac{1}{(\lambda^2 - h^{-2})^2} \norm[\bigg]{\sum_{q \in Q} \beta_{q} Z_{\lambda}^{q}}[L^2(M)]^2 & \leq \frac{1}{((-k + 1) \gamma)^2} \frac{1}{(h^{-1})^2} \norm*{Z_{h_{-k, \gamma}^*, \gamma}^{Q, \beta}}[L^2(M)]^2 \\
& \leq \frac{1}{(k - 1)^2} \frac{C_Q}{\gamma^2} \norm{\beta}[\ell^2(Q)]^2 \frac{ (h^{-1} - k \gamma)^{d-1} }{h^{-2}} \\
& \leq \frac{1}{(k - 1)^2} \frac{C_Q}{\gamma^2} \norm{\beta}[\ell^2(Q)]^2 h^{3-d} \ ,
\end{align*}
where we used that $k \geq 2$ and $1 \leq d$. Meanwhile, on the interval $I_{h, K_h}$,
\[
\sum_{0 \leq \lambda \leq \gamma^2} \frac{1}{(\lambda^2 - h^{-2})^2} \norm[\bigg]{\sum_{q \in Q} \beta_{q} Z_{\lambda}^{q}}[L^2(M)]^2 \leq \frac{h^4}{(1 - h^2 \gamma^2)^2} c_{Q} \norm{\beta}[\ell^2(Q)]^2 \leq \frac{16}{9} c_{Q} \norm{\beta}[\ell^2(Q)]^2 h^{4} \ ,
\]
where we chose $c_{Q} > 0$ such that (recall that $\gamma$ depends on $Q$ but is fixed)
\[
\sum_{0 \leq \lambda \leq \gamma} \norm[\bigg]{\sum_{q \in Q} \beta_{q} Z_{\lambda}^{q}}[L^2(M)]^2 \leq c_Q \norm{\beta}[\ell^2(Q)]^2 \ .
\]
Therefore, we find that for all $0 < h < h_{0}$,
\begin{equation} \label{Eq: Thm large quasimode Gh aux 3}
    \begin{aligned}
    \sum_{k = -K_h}^{-\Upsilon - 1} \sum_{\lambda \in I_{h, k}} \frac{1}{(\lambda^2 - h^{-2})^2} \norm[\bigg]{\sum_{q \in Q} \beta_{q} Z_{\lambda}^{q}}[L^2(M)]^2 & \leq \frac{2C_Q}{\gamma^2} \norm{\beta}[\ell^2(Q)]^2 \frac{h^{3-d}}{\Upsilon} + \frac{16}{9} c_{Q} \norm{\beta}[\ell^2(Q)]^2 h^{4} \\
    & \leq \bigg[ \frac{2C_Q}{\gamma^2} + \frac{16}{9} c_{Q} \bigg] \norm{\beta}[\ell^2(Q)]^2 \frac{h^{3-d}}{\Upsilon} \ .
    \end{aligned}
\end{equation}
A combination of \eqref{Eq: Thm large quasimode Gh aux 1}, \eqref{Eq: Thm large quasimode Gh aux 2}, and \eqref{Eq: Thm large quasimode Gh aux 3} let us conclude the result.

\subsubsection{Proof of \texorpdfstring{\zcref{Thm: quasimodes for GF}(2)}{Theorem quasimodes(2)}}
\label{subsubSec: Proof of Thm quasimodes for GF(2)}
Under the additional assumptions (a) and (b) on $Q$, the width $\gamma$ of the spectral window can be made arbitrarily small thanks to \zcref{Lemma: gamma for density improvement}.

Fix $\gamma \in (0,1)$ and set
\[
\eps_{\gamma} \coloneqq \tfrac{1}{2} \nu_{d} d \gamma^2 > 0 \ .
\]
Let $h_{Q, \gamma^2, \eps_{\gamma}} > 0$ be given by \zcref{Lemma: gamma for density improvement} such that
\[
\abs[\bigg]{ \norm*{Z_{h, \gamma^2}^{Q, \beta} }[L^2(M)]^2 - \nu_{d} d \gamma^2 \ \norm{\beta}[\ell^2(Q)]^2 h^{1-d} } \leq \eps_{\gamma} \ \norm{\beta}[\ell^2(Q)]^2 h^{1-d} \ ,\qquad \forall \, 0 < h < h_{Q, \gamma^2, \eps_{\gamma}} \ .
\]
Set $h_0 = h_{0}(Q, \gamma) \coloneqq \min \{ h_{Q, \gamma^2, \eps_{\gamma}}, \frac{1}{2} \gamma \} > 0$. We decompose $G_{h}^{Q, \beta}$,
\[
G_{h}^{Q, \beta} = \Pi_{I_h(\gamma)} G_{h}^{Q, \beta} + [\Id - \Pi_{I_h(\gamma)}] G_{h}^{Q, \beta} \ .
\]
We will prove there exists $C = C(Q) > 0$ (independent of $\gamma$) such that for all $\beta = (\beta_{q})_{q \in Q} \in \C[N]$, $0 < h < h_0$,
\begin{align*}
    \norm*{\Pi_{I_{h}(\gamma)} G_{h}^{Q, \beta} }[L^2(M)]^2 & \geq \frac{1}{C} \norm{\beta}[\ell^2(Q)]^2 \frac{h^{d-3}}{\gamma^2} \ , \\
    \norm*{G_{h}^{q} - \Pi_{I_{h}(\gamma)} G_{h}^{Q, \beta} }[L^2(M)]^2 & \leq C \norm{\beta}[\ell^2(Q)]^2 \frac{h^{3-d}}{\gamma} \ .
\end{align*}
From this estimates, \eqref{Eq: quasimode gh improvement} follows.

A good lower-bound for the $L^2$ norm of $\Pi_{I_h(\gamma)} G_{h}^{Q, \beta}$ is very simple to obtain thanks to \zcref{Lemma: gamma for density improvement} and our choice of $\eps_{\gamma}$:
\begin{equation} \label{Eq: Thm small quasimode Gh aux 1}
\begin{aligned}
    \norm*{\Pi_{I_h(\gamma)} G_{h}^{Q, \beta} }[L^2(M)]^2 & \geq \sum_{h^{-1} < \lambda \leq h^{-1} + \gamma^2} \frac{1}{(\lambda^2 - h^{-2})^2} \norm[\bigg]{ \sum_{q \in Q} \beta_{q} Z_{\lambda}^{q} }[L^2(M)]^2 \\
    & \geq \frac{1}{\gamma^4} \frac{1}{(2h^{-1} + \gamma^2)^2} \norm*{Z_{h, \gamma^2}^{Q, \beta}}[L^2(M)]^2 \\
    & \geq \frac{1}{\gamma^4} \frac{1}{(3h^{-1})^2} \bigg[ \frac{1}{2} \nu_{d} d \gamma^2 \norm{\beta}[\ell^2(Q)]^2 h^{1-d} \bigg] \\
    & = \frac{1}{18} \nu_{d} d \norm{\beta}[\ell^2(Q)]^2 \frac{h^{3-d}}{\gamma^2} \ .
\end{aligned}
\end{equation}

To find a better upper bound of the $L^2$ norm of $[\Id - \Pi_{I_h(\gamma)}] G_{h}^{Q, \beta}$ we proceed as we did before, splitting into intervals of length $\gamma^2$. Let $K_h \in \N$ be the smallest integer $k$ such that $k \gamma^2 \geq h^{-1}$, and define the collection of intervals
\begin{align*}
    I_{h, k} & \coloneqq (h^{-1} + k \gamma^2, h^{-1} + (k +1) \gamma^2] \qquad k > -K_h \ , \\
    I_{h, K_h} & \coloneqq [0, \gamma^2] \ .
\end{align*}
Note that the length of $I_{h,k}$ is $\gamma^2$ this time. Let $J = \floor{\gamma^{-1}}$ be the greatest integer $j$ such that $j \leq \gamma^{-1}$. We further split $[\Id - \Pi_{I_h(\gamma)}] G_{h}^{Q, \beta}$ into a lower-tail series and a upper-tail series:
\[
\norm*{ [\Id - \Pi_{I_h(\gamma)}] G_{h}^{Q, \beta} }[L^2(M)]^2 \leq \bigg[ \sum_{k = -K_h}^{-J-1} + \sum_{k = J}^{\infty} \bigg] \sum_{\lambda \in I_{h, k}} \frac{1}{(\lambda^2 - h^{-2})^2} \norm[\bigg]{\sum_{q \in Q} \beta_{q} Z_{\lambda}^{q}}[L^2(M)]^2 \ .
\]
For each $k > - K_h$ define $h_{k, \gamma}^* \coloneqq (h^{-1} + k \gamma^2)^{-1}$, so that $I_{h, k} = ( (h_{k, \gamma}^*)^{-1}, (h_{k, \gamma}^*)^{-1} + \gamma^2]$. On each interval $I_{h,k}$, $k \geq J$, we may obtain the following upper bound
\begin{align*}
\sum_{\lambda \in I_{h, k}} \frac{1}{(\lambda^2 - h^{-2})^2} \norm[\bigg]{\sum_{q \in Q} \beta_{q} Z_{\lambda}^{q}}[L^2(M)]^2 & \leq \frac{1}{(k \gamma^2)^2} \frac{1}{(2h^{-1} + k \gamma^2)^2} \norm*{Z_{h_{k, \gamma}^*, \gamma^2}^{Q, \beta}}[L^2(M)]^2 \\
& \leq \frac{1}{k^2 \gamma^4} \frac{1}{(2h^{-1} + k \gamma^2)^2} \bigg[\frac{3}{2} \nu_{d} d \gamma^2 \norm{\beta}[\ell^2(Q)]^2 (h^{-1} + k \gamma^2)^{d-1}\bigg] \\
& \leq \frac{1}{k^2} \frac{3}{2} \nu_{d} d \norm{\beta}[\ell^2(Q)]^2 \frac{h^{3-d}}{\gamma^2} \ ,
\end{align*}
where we used that
\[
\frac{ (h^{-1} + k \gamma^2)^{d-1} }{(2h^{-1} + k \gamma^2)^2} \leq (2h^{-1} + k \gamma^2)^{d-3} = h^{3-d} (2 + k \gamma^2 h)^{d-3} \leq h^{3-d}
\]
because $k \geq 1$ and $1 \leq d \leq 3$. Therefore, we can conclude that
\begin{equation} \label{Eq: Thm small quasimode Gh aux 2}
\begin{aligned}
    \sum_{k = J}^{\infty} \sum_{\lambda \in I_{h, k}} \frac{1}{(\lambda^2 - h^{-2})^2} \norm[\bigg]{\sum_{q \in Q} \beta_{q} Z_{\lambda}^{q}}[L^2(M)]^2 & \leq \frac{3}{2} \nu_{d} d \norm{\beta}[\ell^2(Q)]^2 \frac{h^{3-d}}{\gamma^2} \sum_{k = J}^{\infty} \frac{1}{k^2} \\
    & \leq 3 \nu_{d} d \norm{\beta}[\ell^2(Q)]^2 \frac{h^{3-d}}{\gamma^2} \frac{1}{J}\\
    & \leq 6 \nu_{d} d \norm{\beta}[\ell^2(Q)]^2 \frac{h^{3-d}}{\gamma} \ .
\end{aligned}
\end{equation}
On the other hand, on each interval $I_{h, -k}$, $J + 1 < k < K_h$,
\begin{align*}
\sum_{\lambda \in I_{h, -k}} \frac{1}{(\lambda^2 - h^{-2})^2} \norm[\bigg]{\sum_{q \in Q} \beta_{q} Z_{\lambda}^{q}}[L^2(M)]^2 & \leq \frac{1}{((-k + 1) \gamma^2)^2} \frac{1}{(h^{-1})^2} \norm*{Z_{h_{-k, \gamma}^*, \gamma^2}^{Q, \beta}}[L^2(M)]^2 \\
& \leq \frac{1}{(k - 1)^2\gamma^4} h^2 \bigg[\frac{3}{2} \nu_{d} d \norm{\beta}[\ell^2(Q)]^2 \gamma^2  (h^{-1} - k \gamma)^{d-1} \bigg] \\
& \leq \frac{1}{(k - 1)^2} \frac{3}{2} \nu_{d} d \norm{\beta}[\ell^2(Q)]^2 \frac{h^{3-d}}{\gamma^2} \ ,
\end{align*}
where we used that $k \geq 2$ and $1 \leq d$. Meanwhile, on the interval $I_{h, K_h}$,
\[
\sum_{0 \leq \lambda \leq \gamma^2} \frac{1}{(\lambda^2 - h^{-2})^2} \norm[\bigg]{\sum_{q \in Q} \beta_{q} Z_{\lambda}^{q}}[L^2(M)]^2 \leq \frac{h^4}{(1 - h^2 \gamma^4)^2} c_{Q} \norm{\beta}[\ell^2(Q)]^2 \leq \frac{16}{9} c_{Q} \norm{\beta}[\ell^2(Q)]^2 h^{4} \ ,
\]
where we chose $c_{Q} > 0$ such that
\[
\sum_{0 \leq \lambda \leq 1} \norm[\bigg]{\sum_{q \in Q} \beta_{q} Z_{\lambda}^{q}}[L^2(M)]^2 \leq c_{Q} \norm{\beta}[\ell^2(Q)]^2 \ .
\]
Therefore, we find that
\begin{equation} \label{Eq: Thm small quasimode Gh aux 3}
    \sum_{k = -K_h}^{-\Upsilon - 1} \sum_{\lambda \in I_{h, k}} \frac{1}{(\lambda^2 - h^{-2})^2} \norm[\bigg]{\sum_{q \in Q} \beta_{q} Z_{\lambda}^{q}}[L^2(M)]^2 \leq \frac{3}{2} \nu_{d} d \norm{\beta}[\ell^2(Q)]^2 \frac{h^{3-d}}{\gamma} + \frac{16}{9} c_{Q} \norm{\beta}[\ell^2(Q)]^2 h^{4} \ .
\end{equation}
A combination of \eqref{Eq: Thm small quasimode Gh aux 1}, \eqref{Eq: Thm small quasimode Gh aux 2}, and \eqref{Eq: Thm small quasimode Gh aux 3} let us conclude the result. $\hfill \qed$

\section{Properties of semiclassical defect measures}
\label{Sec: Properties of SDM}

We prove \zcref{Thm: main theorem spectral cond} below, which together with \zcref{Lemma: gamma for density improvement} imply \zcref{Thm: main theorem}. \zcref{Thm: main theorem spectral cond} will be a consequence of \zcref{Thm: structure of eigenfunction of LapL} and \zcref{Thm: quasimodes for GF}. Let $H(x, \xi) \coloneqq \frac{1}{2} \abs{\xi}_x^2$, where $\abs{\xi}_x$ is the norm of $\xi \in T_x^*M$ induced by the Riemannian metric on $M$, be the Hamiltonian on $T^*M$ whose associated Hamiltonian vector field generates the (co-)geodesic flow $\{ \phi_t\}_{t \in \R}$ on $T^*M$. The unit co-sphere bundle $S^*M = \{ (x, \xi) \in T^*M \colon \abs{\xi}_x^2 = 1\}$ is invariant under $\phi_t$.

\begin{Theorem} \label{Thm: main theorem spectral cond}
    Let $(M,g)$ be a closed Riemannian manifold of dimension $2$ or $3$ with positive Laplacian $\Lap$. Let $(\Lap_{Q, L_n})_{n \in \N}$ be a sequence of point-perturbations of $\Lap$ on a fixed finite set $Q \subseteq M$. For every $n \in \N$, let $u_n \in D(\Lap_{Q, L_n}) \subseteq L^2(M)$ and $h_n > 0$ satisfying \eqref{Eq: eigenf main theorem}. Any \SDM{} $\mu$ associated to the sequence $(u_n)_{n \in \N}$ enjoys the following properties:
    \begin{enumerate}
        \item $\mu$ is a probability measure supported in $S^*M$.
        \item Assume that $Q$ is such the following point-wise asymptotics hold.
        \begin{enumerate}
            \item For every $\gamma > 0$, $\eps > 0$ there exists $h_{Q, \gamma, \eps, 1} > 0$ such that \eqref{Eq: diagonal quasimode asymp small gamma} holds for all $0 < h < h_{Q, \gamma, \eps, 1}$ and all $q \in Q$.
            \item For every $\gamma > 0$, $\eps > 0$ there exists $h_{Q, \gamma, \eps, 2} > 0$ such that \eqref{Eq: off-diagonal quasimode asymp small gamma} holds for all $0 < h < h_{Q, \gamma, \eps, 2}$ and all $q, p \in Q$, $p \neq q$.
        \end{enumerate}
        Then $\mu$ is invariant under $\phi_t$.
        \item $\mu$ is invariant under $\phi_t$ outside $Q$ in the following sense: if $\pi(x, \xi) \coloneqq x$, for every $a \in \CinfK(T^*M)$ such that $\pi(\supp a) \subseteq M \setminus Q$,
        \[
        \int_{T^*M} \sympf{a}{H} \mu = 0 \ .
        \]
    \end{enumerate}
\end{Theorem}

\begin{proof}
    For every $n \in \N$, let $w_n \in \ker(\Lap {}- (h_n)^{-2})$, $c_n \in \C$, and $\beta_{n} = (\beta_{n, q})_{q \in Q} \C[N]$ be given by \zcref{Thm: structure of eigenfunction of LapL}:
    \begin{equation*}
        u_n = w_n + c_n \, g_{h_n}^{Q, \beta_n} \ , \qquad g_{h_n}^{Q, \beta_n} \coloneqq \frac{1}{\norm*{G_{h_n}^{Q, \beta_n}}[L^2(M)]} G_{h_n}^{Q, \beta_n} \ .
    \end{equation*}
    Up to extraction of a subsequence, we may assume that the sequence $(u_n)_{n \in \N}$ has a unique \SDM{} $\mu$ along $(h_n)_{n \in \N}$. Observe that $\norm{w_n}[L^2(M)] \leq 1$, $\abs{c_n} \leq 1$ for all $n \in \N$ because $w_n$ and $g_{h_n}^{Q, \beta_n}$ are orthogonal for all $n \in \N$. Denote $v_n \coloneqq c_n \, g_{h_n}^{Q, \beta_n}$.

    \textbf{Proof of items (1) and (2)}. Recall the notation $I_{n}(r) = ((h_n)^{-1} - r, (h_n)^{-1} + r]$ and $\Pi_{I} = \charf_{I}(\sqLap)$. For every $r > 0$ one has
    \begin{equation} \label{Eq: Lap on quasimodes estimate}
        [h_n^2 \Lap {}- 1] \Pi_{I_n(r)} = r \BigO_{L^2 \to L^2}(h_n) \ .
    \end{equation}
    Moreover,
    \[
    \norm*{u_n - \Pi_{I_n(r)}u_n}[L^2(M)] = \norm*{v_n - \Pi_{I_n(r)}v_n}[L^2(M)] \ ,
    \]
    therefore, for every $b \in \CinfK(T^*M)$ and every $n \in \N$,
    \begin{equation} \label{Eq: controlUps}
        \begin{aligned}
            \abs*{ \ip*{u_n}{\Op[h_n]{b} u_n}[L^2(M)] - \ip{\Pi_{I_n(r)}u_n}{\Op[h_n]{b} \Pi_{I_n(r)}u_n}[L^2(M)] } & \\
            & \hspace{-100pt} \leq 2 \norm{\Op[h_n]{b}}[L^2 \to L^2] \, \norm{v_n - \Pi_{I_n(r)}v_n}[L^2(M)] \ .
        \end{aligned}
    \end{equation} 
    \zcref{Thm: quasimodes for GF} will allow us to control the error $v_n - \Pi_{I_n(r)} v_n$ uniformly in $n \in \N$, and therefore to obtain information on the \SDM{} $\mu$ from the quasimodes $\Pi_{I_n(r)} u_n$.

    Let us first show that $\mu$ is supported in $S^*M$. Let $C > 0$, $\gamma > 0$ and $h_0 > 0$ be given by \zcref{Thm: quasimodes for GF}(i). Take $N_0 \in \N$ such that $0 < h_n < h_0$ for all $n \geq N_0$. This result ensures that for every $\Upsilon \in \N$, $n \geq N_0$,
    \[
    \norm*{v_n - \Pi_{I_n(\Upsilon \gamma)}v_n}[L^2(M)] \leq C \Upsilon^{-1/2} \ .
    \]
    Therefore, \eqref{Eq: controlUps} gives in this case that there exists $M_b > 0$ such that for all $n \geq N_0$,
    \begin{equation}\label{Eq: bdu1}
        \abs*{ \ip*{u_n}{\Op[h_n]{b} u_n}[L^2(M)] }  \leq M_b \, \Upsilon^{-1/2} + \abs*{ \ip*{ \Pi_{I_n(\Upsilon \gamma)} u_n }{\Op[h_n]{b} \Pi_{I_n(\Upsilon \gamma)} u_n }[L^2(M)] } \ .
    \end{equation}
    Fix $a \in \CinfK(T^*M)$. Since the symbolic calculus of semiclassical pseudodifferential operators gives
    \[
    \Op[h]{a (2H - 1)} = \Op[h_n]{a} [h_n^2 \Lap {} - 1] + \BigO_{L^2 \to L^2}(h_n) \ ,
    \]
    we can conclude that
    \begin{multline*}
        \ip*{ \Pi_{I_n(\Upsilon \gamma)} u_n }{\Op[h_n]{a(2H - 1)} \Pi_{I_n(\Upsilon \gamma)} u_n}[L^2(M)] \\
        = \ip*{ \Pi_{I_n(\Upsilon \gamma)} u_n }{\Op[h_n]{a} [h_n^2 \Lap {}- 1] \Pi_{I_n(\Upsilon \gamma)} u_n}[L^2(M)] + \BigO(h_n) = \Upsilon \BigO(h_n) \ .
    \end{multline*}
    Applying \eqref{Eq: bdu1} with the test symbol $b = a(H-1)$ and taking limits as $n \to \infty$ gives:
    \[
    \abs*{ \int_{T^*M} a(x, \xi) (2H(x, \xi) - 1) \mu(\D{x}, \D{\xi}) } \leq M \Upsilon^{-1/2} \qquad \forall \ \Upsilon \in \N \ ,
    \]
    and thus
    \[
    \int_{T^*M} a(x, \xi) (\abs{\xi}_x^2 - 1) \mu(\D{x}, \D{\xi}) = 0 \ .
    \]
    Since $a \in \CinfK(T^*M)$ was arbitrary, this implies $\supp \mu \subseteq S^*M$.

    Assume now that $Q$ is such that estimates \eqref{Eq: diagonal quasimode asymp small gamma} and \eqref{Eq: off-diagonal quasimode asymp small gamma} hold as stated, and let us show that $\mu$ is invariant under the geodesic flow. Fix $\gamma > 0$ and let $C = C(Q) > 0$ and $h_{0} = h_0(Q, \gamma) > 0$ be given by \zcref{Thm: quasimodes for GF}(ii). Take $N_0 \in \N$ such that $0 < h_n < h_0$ for all $n \geq N_0$. This time, \eqref{Eq: controlUps} with $r = \gamma > 0$ gives the existence of $M_b > 0$ such that for any $n \geq N_0$
    \begin{equation}\label{Eq: bdu2}
        \abs*{ \ip*{u_n}{\Op[h_n]{b} u_n}[L^2(M)] } \leq M_b \gamma^{1/2} + \abs*{ \ip*{ \Pi_{I_n(\gamma)} u_n }{\Op[h_n]{b} \Pi_{I_n(\gamma)} u_n }[L^2(M)] } \ .   
    \end{equation}
    Thanks to the symbolic calculus,
    \begin{equation} \label{Eq:symb-calculus-commutator}
        \Op[h]{\sympf{a}{H}} = \frac{i}{2h_n} \big[ \Op[h_n]{a} , \, h_n^2 \Lap \big] + \BigO_{L^2 \to L^2}(h_n)    \ ,
    \end{equation}
    and we can conclude with $b = \sympf{a}{H}$ that
    \begin{multline*}
        \ip*{ \Pi_{I_n(\gamma)} u_n }{\Op[h_n]{\sympf{a}{H}} \Pi_{I_n(\gamma)} u_n}[L^2(M)] = \\
        = \frac{i}{2h_n} \ip*{ \Pi_{I_n(\gamma)} u_n }{ \big[ \Op[h_n]{a} , \, h_n^2 \Lap \big] \Pi_{I_n(\gamma)} u_n}[L^2(M)] + \BigO(h_n) = \gamma \BigO(1) + \BigO(h_n) \ .
    \end{multline*} 
    Taking limits $n \to \infty$ and $\gamma \to 0^+$ in \eqref{Eq: bdu2} with test symbol $b = \sympf{a}{H}$ shows that
    \[
    \int_{T^*M} \{ a ,\, H \} (x, \xi) \, \mu(\D{x}, \D{\xi}) = 0 \ .
    \]
    Since $a \in \CinfK(T^*M)$ was arbitrary, this implies that $(\phi_t)_* \mu = \mu$ for every $t \in \R$.

    \textbf{Proof of item (3)}. The proof of this statement is purely semiclassical and does not depend on the approximations obtained in \zcref{Sec: Approximation by quasimodes}. It relies on the observation that eigenfunctions $u_n$ are smooth on $M \setminus Q$ (elliptic regularity), thus $\Lap_{Q, L_n} \rchi u_n = \Lap \rchi u_n$ for every $\rchi \in \CinfK(M \setminus Q)$.

    Take $a \in \CinfK(T^*M)$ such that $\pi(\supp a) \subseteq M \setminus Q$. Let $\rchi \in \CinfK(M \setminus Q, [0, 1])$ be such that
    \[
    \rchi \equiv 1 \quad \text{on} \quad \pi (\supp a) \ .
    \]
    Since $\rchi^2 \sympf{a}{2H - 1} = \sympf{a}{2H - 1}$, symbolic calculus and the identity $\Op[h_n]{\rchi} = M_{\rchi}$ give
    \begin{align*}
        \ip*{u_n}{\Op[h_n]{ \sympf{a}{2H - 1} } u_n}[L^2(M)] & = \ip*{u_n}{\Op[h_n]{\rchi^2 \sympf{a}{2H - 1}} u_n}[L^2(M)] \\
        & = \ip*{\rchi u_n}{\Op[h_n]{\sympf{a}{2H - 1}} \rchi u_n}[L^2(M)] + \BigO(h_n) \ .
    \end{align*}
    Since $\rchi u_n \in \CinfK(M \setminus Q)$ and $\norm{\rchi u_n}[L^2(M)] \leq \norm{u_n}[L^2(M)] = 1$, we can apply \eqref{Eq:symb-calculus-commutator} to find
    \begin{multline*}
        \ip*{u_n}{\Op[h_n]{ \sympf{a}{2H - 1} } u_n}[L^2(M)] = \\
        \frac{i}{h} \Big[ \ip*{ \Op[h_n]{a} [h_n^2 \Lap {}- 1] (\rchi u_n) }{ \rchi u_n }[L^2(M)] - \ip*{ \rchi u_n }{ \Op[h_n]{ a} [h_n^2 \Lap {}- 1] (\rchi u_n) }[L^2(M)] \Big] + \BigO(h_n) \ .
    \end{multline*}
    Next, we observe that $[h_n^{2} \Lap {}- 1] (\rchi u_n)$ is a smooth function supported away from $\pi(\supp a)$, therefore
    \[
    \lim_{n \to \infty} \norm*{\Op[h_n]{a} [h_n^2 \Lap {}- 1] (\rchi u_n) }[L^2(M)] = 0 \ .
    \]
    Putting everything together we get
    \[
    \lim_{n \to \infty} \ip*{u_n}{\Op[h_n]{ \sympf{a}{H} } u_n}[L^2(M)] = 0 \ . \qedhere
    \]
\end{proof}

\printbibliography

\end{document}

%% file: Repository/commands.tex
\def \eps{ \varepsilon }
\DeclareDocumentCommand{\R}{O{} O{}}{ \mathbb{R}_{#2}^{#1} }
\DeclareDocumentCommand{\C}{O{} O{}}{ \mathbb{C}_{#2}^{#1} }
\DeclareDocumentCommand{\S}{O{} O{}}{ \mathbb{S}_{#2}^{#1} }
\DeclareDocumentCommand{\T}{O{} O{}}{ \mathbb{T}_{#2}^{#1} }
\DeclareDocumentCommand{\Z}{O{} O{}}{ \mathbb{Z}_{#2}^{#1} }
\def \N{\mathbb{N}}

\def \Cinf{\mathcal{C}^{\infty}}
\def \CinfK{\mathcal{C}_{c}^{\infty}}

\def \charf{\mathds{1}}
\def \fdot{\, \cdot \,}

\def \BigO{\mathcal{O}}

\DeclarePairedDelimiter{\abs}{\lvert}{\rvert}

\NewDocumentCommand{\norm}{s O{} m O{}}
    {
        \IfBooleanTF{#1}
            {\ensuremath{{\left\lVert #3 \right\rVert}_{#4}}}
            {{#2\lVert {#3} #2\rVert}_{#4}}
    } 

\NewDocumentCommand{\ip}{s O{} m m O{}}
    {
        \IfBooleanTF{#1}
            {\ensuremath{{\left\langle #3 \, \middle| \, #4 \right\rangle}_{#5}}}
            {{#2\langle #3 \, #2| \, #4 #2\rangle}_{#5}}
    }

\DeclarePairedDelimiter{\bra}{\langle}{|}
\DeclarePairedDelimiter{\ket}{|}{\rangle}

\DeclarePairedDelimiter{\floor}{\lfloor}{\rfloor}

\DeclareRobustCommand{\rchi}{{\mathpalette\irchi\relax}} 
\newcommand{\irchi}[2]{\raisebox{\depth}{$#1\chi$}} 

\newcommand*{\cl}[1]{\mkern 2mu\overline{\mkern-1.5mu#1\mkern-1.5mu}\mkern 2mu} 
\newcommand*{\conj}[1]{ \mkern 2mu \overline{\mkern-1mu #1 \mkern-1mu} \mkern 2mu}
\newcommand*{\D}[1]{\mathop{\mathrm{d}#1}}

\DeclareMathOperator{\spn}{Span}
\DeclareMathOperator{\Id}{Id}
\NewDocumentCommand{\Span}{O{} O{} m}{\spn_{#1} #2\{ #3 #2\}}

\let\Re\relax
\let\Im\relax
\DeclareMathOperator{\Re}{\mathfrak{Re}}
\DeclareMathOperator{\Im}{\mathfrak{Im}}

\DeclareMathOperator{\dom}{dom}

\DeclareMathOperator{\Lap}{\Delta}
\DeclareMathOperator{\sqLap}{\sqrt{\Lap}}
\DeclareMathOperator{\supp}{supp} 
\DeclareMathOperator{\Opp}{Op}
\DeclareMathOperator{\vol}{vol}
\DeclareMathOperator{\Spec}{Spec}
\DeclareMathOperator{\card}{card}
\let\div\relax
\DeclareMathOperator{\div}{div}

\NewDocumentCommand{\Op}{O{} O{} m}{\ensuremath{\Opp_{#1}^{#2} \left( #3 \right)}}